\documentclass[a4paper,reqno,12pt]{amsart}

\usepackage[utf8]{inputenc}
\usepackage{amsmath, amssymb, amsthm, epsfig}
\usepackage{hyperref, latexsym}
\usepackage{url}
\usepackage[mathscr]{euscript}
\usepackage{color}
\usepackage{harpoon}
\usepackage{url}
\usepackage{mathabx}
\usepackage{tikz}

\usepackage{fullpage} 
\usepackage{setspace}
\usepackage{mathtools}
\mathtoolsset{showonlyrefs}

\def\today{\ifcase\month\or
  January\or February\or March\or April\or May\or June\or
  July\or August\or September\or October\or November\or December\fi
  \space\number\day, \number\year}
\DeclareMathOperator{\sgn}{\mathrm{sgn}}

\DeclareMathOperator{\supp}{\mathrm{supp}}

\newtheorem{theorem}{Theorem}

\newtheorem{lemma}[theorem]{Lemma}
\newtheorem{proposition}[theorem]{Proposition}
\newtheorem{remark}{Remark}
\newtheorem{definition}{Definition}

\newcommand{\A}{\mathcal{A}}
\newcommand{\B}{\mathcal{B}}

\newcommand{\E}{\mathcal{E}}

\renewcommand{\S}{\mathcal{S}}
\newcommand{\T}{\mathcal{T}}

\renewcommand{\H}{\mathbb{C}^+}

\newcommand{\n}{\mathbb{N}}
\newcommand{\z}{\mathbb{Z}}
\newcommand{\q}{\mathbb{Q}}
\renewcommand{\r}{\mathbb{R}}
\newcommand{\cp}{\mathbb{C}} 
\newcommand{\im}{{\rm Im}\,}
\newcommand{\re}{{\rm Re}\,}
\newcommand{\ft}{\widehat}

\newcommand{\bo}{\boldsymbol}

\newcommand{\wt}{\widetilde}
\newcommand{\br}{{\bo r}}

\newcommand{\la}{\lambda}
\newcommand{\ga}{\gamma}
\newcommand{\al}{\alpha}
\newcommand{\be}{\beta}
\newcommand{\ep}{\varepsilon}

\newcommand{\si}{\sigma}

\newcommand{\p}{\varphi}
\renewcommand{\d}{\mathrm{d}}
\newcommand{\w}{\ov{w}}
\newcommand{\1}{{\bf 1}}
\newcommand{\apc}{\mathrm{AP}(\H)}
\newcommand{\fs}{\mathrm{FS}}
\newcommand{\afs}{\mathrm{AFS}}
\renewcommand{\E}{\mathbb{E}}
\renewcommand{\emph}[1]{{\bf #1}}
\newcommand{\ov}[1]{\overline{#1}}
\newcommand{\apr}{{\rm AP}(\r)}
\newcommand{\bt}{{\rm BT}}
\newcommand{\udps}{{\rm UDPS}}
\newcommand{\rrtp}{\mathrm{RRTP}}
\newcommand{\cm}{\rm CM}
\newcommand{\acm}{\rm ACM}
\renewcommand{\1}{{\bo 1}}
\newcommand{\del}{{\bo \delta}}
\renewcommand{\Re}{{\rm Re}\,}
\renewcommand{\Im}{{\rm Im}\,}
\newcommand{\res}[1]{ \mathop{\mathrm{Res}}_{z = #1}\,}
\newcommand{\spec}{{\rm spec}}
\renewcommand{\a}{a(\cdot)}
\newcommand{\h}{\mathfrak{H}}

\begin{document}


\title[]{A classification of Fourier summation formulas and crystalline measures}
\author[Gon\c{c}alves]{Felipe Gon\c{c}alves}
\date{\today}
\subjclass[2010]{}
\address{The University of Texas at Austin, 2515 Speedway, Austin, TX 78712, USA  \newline   {\&} IMPA - Instituto de Matemática Pura e Aplicada, Rio de Janeiro, 22460-320, Brazil.}
\email{felipe.ferreiragoncalves@austin.utexas.edu}
\email{goncalves@impa.br}
\allowdisplaybreaks


\begin{abstract}
We completely classify Fourier summation formulas, and in particular, all crystalline measures with quadratic decay. Our classification employs techniques from almost periodic functions, Hermite-Biehler functions, de Branges spaces and Poisson representation. We show how our classification generalizes recent results of Kurasov \& Sarnak and Olevskii \& Ulanovskii. As an application, we give a new classification result for nonnegative measures with uniformly discrete support that are bounded away from zero on their support. Moreover, we give a new construction using eta-quotients, generalizing an old example of Guinand.
\end{abstract}


\maketitle


\section{Introduction}

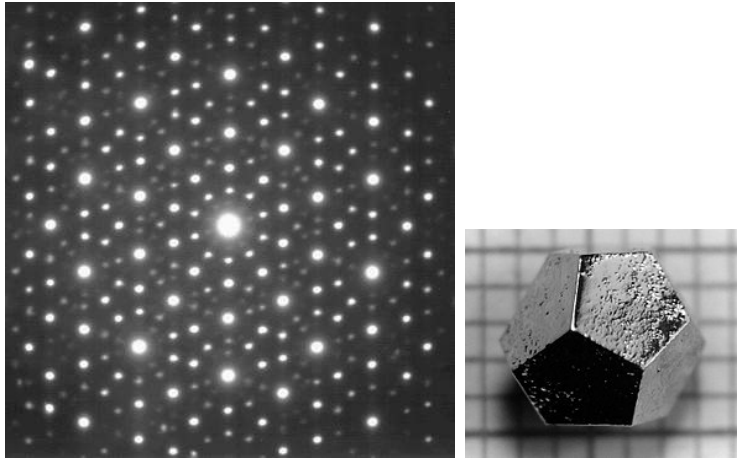
To understand the atomic structure of a crystal (of some material) one can fire a beam of electrons into the object and observe an electron diffraction pattern on a screen\footnote{This technique is called high resolution electron microscopy (HRTEM), but another common technique is X-ray microscopy.}. In most materials one observes a periodic diffraction pattern, such as a  two-dimensional lattice with $6$-fold symmetry, as in Tantalum pentoxide ${\rm Ta}_2{\rm O}_5$. However, this is not the case with Holmium-magnesium-zinc Ho-Mg-Zn, where one observes some $5$-fold and $10$-fold symmetries, hence a non-periodic diffraction pattern (see \cite[Figures 1 and 2]{dan}). 
The first to record such phenomena experimentally was materials scientist Dan Shechtman, in 1982, and today these materials are called quasicrystals. He was awarded the Nobel Prize in Chemistry in 2011 for his breakthrough. Mathematically, if $\mu$ represents the crystal (so $\mu\geq 0$ and $\supp(\mu)$ is uniformly discrete) then what we see in the diffraction pattern is effectively $|\ft \mu|^2$ (where $\ft \mu$ is the Fourier transform of $\mu$). A famous inverse problem is that of reconstructing atomic structure from diffraction data, overcoming the issue that diffraction only determines intensity of spots but loses phase information. Real-world quasicrystals (such as Decagonite $\text{Al}_{71}\text{Ni}_{24}\text{Fe}_5$), although uniformly discrete in physical space, typically have everywhere dense diffraction spectrum\footnote{We thank Michael Baake for pointed this out, among several other remarks.}.  The diffraction pattern is a super-position
of Dirac combs with several
intensities (Bragg peaks) however, due to instrument precision and tuning, we only see peaks that have, say, an intensity above $10^{-4}$ of the brightest peak. Hence, in real physical experiments, although the diffraction picture visually seems to be locally finite, nevertheless the real measure $\ft \mu$  has a dense support. On the other hand, stable $1$D quasicrystals with locally finite spectra do exist in nature, and can be found as projections along pseudo axis's, see \cite[Figure 1]{QS1D} and the references in \cite[Ch. 7]{DW}. Nowadays there is unified framework to construct measures $\mu$ whose spectrum fits the data observed in the diffraction patterns of physical quasicrystals, these are called model sets (or Meyer sets) \cite{BG2,DW,La,Me95,Me16,Mo}.


From the mathematical standpoint, many mathematicians have wondered about what different types of measures $\mu$ with pure point support and spectrum are there, and what kind of geometrical properties of the spectrum are allowed under various conditions on $\mu$? The main classification question from the mathematical perspective is:
\begin{center}
\it Can on one classify all measures $\mu$ with pure point support and spectrum?
\end{center}
Freeman Dyson addressed this question in 2009 in his famous paper \cite{Dy}, using the word quasicrystals for such measure,  suggesting that such classification might be a very difficult mathematical problem, but made the ``outrageous suggestion" it could potentially give a proof of the Riemann Hypothesis\footnote{We do not share the same belief.}. He said: ``My suggestion is the following. Let us pretend that we do not know that the Riemann Hypothesis is true. Let us tackle the problem from the other end. Let us try to obtain a complete enumeration and classification of one-dimensional quasicrystals. That is to say, we enumerate and classify all point distributions that have a discrete point spectrum."  Dyson stated this because of Guinand's prime summation formula (also known as Weyl's summation formula), which we explain in the Section \ref{sec:exemp}. The notorious paper of Guinand \cite{Gui}, which inspired Dyson, grasps such classification, but it was forgotten awhile. Meyer resurrected Guinand's paper in \cite{Me16}.

We say a measure $\mu$ is \emph{crystalline} if it has locally finite support and its Fourier transform $\ft \mu$ also has locally finite support\footnote{Such measures are sometimes also called quasicrystals, Fourier quasicrystals, or doubly sparse, depending on the author, area, context, decay conditions, etc. We will stick with Meyer's denomination of crystalline measure.}. In this scenario, the first geometric classification breakthrough occurred in the work of Lev \& Olevskii \cite{LO,LO2}, where they classify classical Poisson summation as the only crystalline measures that have uniformly discrete support and spectrum. Guinand's example \cite[p. 265]{Gui} shows the theorem is tight (we generalize his construction in Section \ref{sec:exemp}). Recently, the field is very active, and the amazing works of Radchenko \& Viazovska \cite{RV},  Ramos \& Sousa \cite{RS1}, Kurasov \& Sarnak \cite{KS}, Olevskii \& Ulanovskii \cite{OU} and Alon et. al. \cite{AV,ACV}, but also, the not so recent works of Bohr \cite{Bo}, Besicovitch \cite{Be}, Guinand \cite{Gui}, de Branges \cite{dB} and Meyer \cite{Me}, deeply inspired this paper.

In this paper we completely classify (what we call) Fourier summation formulas, and in particular we classify crystalline measures that have quadratic decay, although our results are more general and can be applied to dense spectra. As far as we know, it is the first time classification results are proven in such generality.  Our classification involves heavily the theory of almost periodic analytic functions \cite{Be} and de Branges spaces techniques \cite{dB}. Almost periodic functions have permeated this are for some time now, for example in the works of Guinand \cite{Gui}, Meyer \cite{Me12} and Baake et. al. \cite{BS, BG2, BST}, among others. On the other hand, the connection with de Branges spaces that we make  here is totally new.

A \emph{Fourier summation formula} should be one where the identity
$$
\int_\r \ft \p(t)\d\mu(t) = \sum_{n\in \n} a(\la_n)\p(\la_n)
$$
holds, at the bear minimum, for all functions $\p \in C^\infty_c(\r)$, where $\{\la_n\}_{n\in \n}\subset \r$ is some sequence of real nodes, $\{a(\la_n)\}_{n\in \n}\subset \cp$ are some complex coefficients and  $\mu$ is some complex-valued measure. There is a zoo of formulas of such type which we describe in Section \ref{sec:exemp}. We choose the above form of because: (1) We want a summation formula, so there must be a sum somewhere in the definition and such summation must be singular with respect to Lebesgue measure; (2)  Allowing the right hand side above to have a vanishing singular part is boring, as otherwise any $L^1$-function would give rise to a Fourier summation formula. 
\vspace{-1mm}
\subsection{Overview} Theorems \ref{thm:maineven} and \ref{cor:1} are our main general classification results and we later specialize them for measures with locally finite support (in particular, also crystalline measures) in Theorems \ref{thm:HBmupos} and \ref{thm:HBmudisc}. Finally, we derive a new characterization Theorem \ref{thm:5} for measures with uniformly discrete support such that $\mu|_{\supp(\mu)}\geq \ep$, but with no geometric assumption on $\ft \mu$, which perhaps makes this result most useful to physics.

\section{Main Results}
We will have to introduce several things before we can state our first main result. We say that a complex-valued measure $\mu$ on $\r$ is \emph{locally finite} if its total variation 
$$
|\mu|(K) = \sup_{\sqcup_{j\in\n} P_j = K} \sum_{j\in\n} |\mu(P_j)| < \infty
$$
is finite on every compact set $K\subset \r$, or equivalently, if there are nonnegative locally finite Borel measures $\mu_{j}$ such that $\mu=\mu_1-\mu_2 + i(\mu_3-\mu_4)$. We say that $\mu$ is \emph{tempered} if it is a complex-valued locally finite measure and there exists a tempered distribution $u\in \S'(\r)$ such that $u(\p)=\int_\r \p \d\mu$ for all $\p\in C^\infty_c(\r)$. Some authors might not concord with this definition, but we are following the setup of \cite{BS}. We say  $\mu$ is \emph{strongly tempered} if $\p\in L^1(|\mu|)$ for all $\p\in \S(\r)$. By \cite[Prop. 2.5]{BS}, this is equivalent to the  \emph{degree} of $\mu$
\begin{align}\label{eq:decaymu}
\deg(\mu):=\min \bigg\{ n\in \z : \int_\r \frac{\d|\mu|(t)}{(1+t^2)^{n/2}}<\infty\bigg\}
\end{align}
being $<\infty$. In particular, the functional $\p \mapsto \int_\r \p \d\mu$ defines a tempered distribution, hence $\mu$ is tempered.  The implication [tempered] $\Rightarrow$ [strongly tempered] is false. Indeed, the measure
$
\mu = \sum_{n\geq 0} 3^{n^2}(\del_{n}-\del_{n+4^{-n^2}})
$
is not strongly tempered but it is tempered, since $\int 2^{-t^2}\d |\mu|(t)\geq  \sum_{n\geq 0} 3^{n^2}2^{-n^2}=\infty$ but\footnote{Whenever it is convenient, we write $A\ll B$ to mean that there is $C>0$ such that $|A|\leq C|B|$.} $|\sum_{n\geq 0} 3^{n^2}(\p({n})-\p({n+4^{-n^2}}))|\ll \|\p '\|_\infty$. 

We say that a function $a:\r\to\cp$ is \emph{locally summable} if $\supp(a):=\{\la\in \r : a(\la)\neq 0\}$ is at most countable and for some enumeration $\supp(a)=\{\la_n\}_{n\geq 1}$ we have 
$$\sum_{\la_n \in [-T,T]} |a(\la_n)|<\infty,$$
for all $T>0$. This is equivalent to say that $\nu=\sum_{n\geq 1} a(\la_n)\del_{\la_n}$ is locally finite. In particular, local sums $\sum_{\la \in K} a(\la)$ are always well-defined for any bounded set $K\subset \r$. 

\begin{definition}[\emph{Fourier Summation Pairs}]
We say $(\mu,a)$ is a Fourier summation pair $(\fs$-pair$)$ if $\mu$ is a strongly tempered measure, $a:\r\to\cp$ is {locally summable} and
$$
\int_\r \ft \p(t)\d\mu(t) = \sum_{\la\in\r} a(\la)\p(\la)
$$
for every $\p\in C^\infty_c(\r)$. Equivalently, $\mu$ is a strongly tempered measure and for some locally summable function $a:\r\to\cp$ we have $\ft \mu|_{C^\infty_c}=\sum_{\la\in\r} a(\la)\del_{\la}$. 
\end{definition}

\noindent  We use the following definition for \emph{Fourier transform}:
$$
\ft \p(\xi) = \int_\r \p(x)e^{-2\pi i x \xi}\d x.
$$
Perhaps the two most simple examples of $\fs$-pairs are
\begin{align*}
 \int_\r \ft \p(t)\d t = \p(0) \ \ (\text{Fourier Inversion}) \ \ \text{ and }  \ \ \sum_{n\in \z} \ft \p(n) = \sum_{n\in \z} \p(n)  \ \ (\text{Poisson Summation}).
\end{align*}

\begin{remark}[Real-Antipodal splitting]
In the following we show that any $\fs$-pair can be split into two real-antipodal $\fs$-pairs. We say that a function $a:\r \to \cp$ is \emph{antipodal} if $a(-\la)=\ov{a(\la)}$ for every $\la\in\r$. We say that an $\fs$-pair $(\mu,a)$ if \emph{real-antipodal} is $\mu$ is real-valued and $\a$ is (and must be) antipodal.  Indeed, if we are given an arbitrary $\fs$-pair $(\mu,a)$ then we can write 
$
a_1(\la) = \frac{a(\la)+\ov{a(-\la)}}{2}, \ \ a_2(\la)=\frac{\ov{a(-\la)}-a(\la)}{2i},
$
and so $a=a_1-ia_2$, with each $a_j$ antipodal, and we can write 
$\mu= \mu_1 - i\mu_2$, where $\mu_1=\Re \mu$ and $\mu_2=\Im \mu$, and so each $\mu_j$ is real-valued. If $\p\in C^\infty_c(\r)$ is antipodal then $\ft \p$ is real, and if we further assume that either $\p$ is real-valued or imaginary-valued, then from the identity
$$
\sum_{\la\in \r} (a_1(\la)-ia_2(\la))\p(\la) = \int_\r \ft \p(t)\d \mu_1(t) -i \int_\r \ft \p(t)\d \mu_2(t) 
$$
and the fact that $\sum_{\la\in \r} a_j(\la)\p(\la) $ is real, one deduces that
$$
\sum_{\la\in \r} a_j(\la)\p(\la) = \int_\r \ft \p(t)\d \mu_j(t)
$$
for $j=1,2$. Since for any $\p\in C^\infty_c(\r)$ we have $\p=\p_1-i\p_2$ where each $\p_j$ is antipodal, and each $\p_j=\Re \p_{j}- i \Im \p_{j}$, where $\Re \p_{j}$ is real and even and $i \Im \p_{j}$ is imaginary and odd (hence both are antipodal as well), we deduce by linearity that $(\mu_j,a_j)$ is a real-antipodal $\fs$-pair for $j=1,2$. As an example, note that
$
\sum_{n \in \z} i^n \ft \p(n) = \sum_{n\in \z} \p(n+1/4)
$
has the real-antipodal splitting 
\begin{align*}
\sum_{n \in \z} (-1)^n \ft \p(2n)  &= \frac12 \sum_{n\in \z} \left(\p(n+1/4) +  \p(n-1/4)\right)\\ 
\sum_{n \in \z} (-1)^n \ft \p(2n+1) & =  \frac{1}{2i}\sum_{n\in \z} \left(\ft \p(n+1/4) - \ft \p(n-1/4)\right).
\end{align*}
\end{remark}

We say that a locally summable function $a:\r\to\cp$ has \emph{exponential growth} if
$$
\sum_{\la\in \r} |a(\la)|e^{-c|\la|}<\infty
$$
for some $c>0$. We let $\H=\{z=x+iy : y>0\}$. Following Besicovitch\footnote{Besicovitch considers the half-plane $\re z>0$ instead of $\Im z>0$, so one has to mentally perform a rotation in order to use his results.} \cite[Ch. 3]{Be}, we say a holomorphic function $f:\H\to\cp$ is \emph{almost periodic} if for every $0<\al<\be<\infty$ and $\ep>0$ there is a relatively dense\footnote{A set $\tau\subset \r$ is relatively dense if there exists $l>0$ such that $(x,x+l)\cap \tau \neq \emptyset$ for every $x\in\r$.} set of $\ep$-translations $\tau \subset \r$
such that 
$$
\sup_{\al < \Im z < \be} |f(z)-f(z+t)|\leq \ep
$$ 
for every $t\in \tau$. We write $f\in \apc$. If $f\in \apc$ it can be shown that 
$$
\E f(\la) := \lim_{T\to\infty} \frac{1}{2T}\int_{-T+iy}^{T+iy} f(z)e^{-2\pi i \la z}\d z
$$
exists for all $\la\in\r$ and is independent of $y>0$ (see Section \ref{sec:AP}). It is straightforward to see that if instead $f(\cdot +ic)\in \apc$ for some $c>0$, then the above limit exists and is independent of $y>c$. If $f\in \apc$ we define the \emph{spectrum} of $f$ by
$$
\spec(f):=\{\la\in \r : \E f(\la)\neq 0\},
$$ 
and it can be shown that this set is at most countable. We say a holomorphic function $g:\H\to\cp$ is of \emph{bounded type}\footnote{This is the same to say that $\log |g(z)|$ has a nonnegative harmonic majorant.} if there are two bounded holomorphic functions $P,Q:\H \to\cp$ such that $g=P/Q$. The following is the first main result of this paper.
\begin{theorem}\label{thm:maineven}
Let $(\mu,a)$ be a real-antipodal $\fs$-pair. Assume that $\deg(\mu)\leq 2$ and that $\a$ has exponential growth. Then we have:
\begin{enumerate}
\item[(I)] The limit
\begin{align}\label{eq:fdef}
f(z) =\frac{1}2 a(0) +  \lim_{T\to \infty} \sum_{0<\la<T} a(\la) (1-\la/T) e^{2\pi i \la z}
\end{align}
converges uniformly in compact sets of $\H$;
\item[(II)] There is $c_1>0$ such that $f(\cdot+ic_1) \in \apc$;
\item[(III)] There are $c_2,B>0$, with $c_2\geq c_1$, such that for any trigonometric polynomial $p(x)=\sum_{n=1}^N p_n e^{2\pi i \theta_n x}$, with $\{\theta_n\} \subset [0,\infty)$, we have
$$
\limsup_{T\to\infty} \bigg|\frac{1}{2T}\int_{-T}^T f(x+ic_2)\ov{p(x)}\d x\bigg| \leq B \max |p_n|;
$$
\item[(IV)]  $\exp(f)$ is of bounded type.
\end{enumerate}
Conversely, suppose $f:\H \to \cp$ is a given holomorphic function such that items $({\rm II}), ({\rm III})$ and $({\rm IV})$ hold true. Then there is $c\geq 0$ such that the limit
$$
\E f(\la) := \lim_{T\to\infty} \frac{1}{2T}\int_{-T+iy}^{T+iy} f(z)e^{-2\pi i \la z}\d z
$$
exists for every $\la\in\r$ and $y>c$, is independent of $y$, vanishes identically for $\la<0$ and defines a locally summable function with exponential growth such that
$$
\sum_{\la\in \r} |\E f(\la)|e^{-2\pi y |\la|} < \infty
$$
for every $y>c$. Moreover, there is a unique real-valued locally finite measure $\mu$ of degree at most $2$ satisfying
\begin{align}\label{eq:repkerLphiid}
\frac{f(z)+\ov{f(w)}}{z-\w} = \frac{1}{2\pi i} \int_\r \frac{\d \mu(t)}{(z-t)(\w -t)} \quad 
\end{align}
for all $w,z\in \H$. Furthermore, if we let
\begin{align}\label{eq:adef}
a(\la)=\begin{cases} \E f(\la) &\mbox{if } \la>0 \\ 
2\Re \E f(0) &\mbox{if } \la=0 \\
 \ov{\E f(-\la)} &\mbox{if } \la<0
\end{cases}
\end{align}
then $(\mu,a)$ is an real-antipodal $\fs$-pair and identity \eqref{eq:fdef} holds.
\end{theorem}
The proof we give for Theorem \ref{thm:maineven} in the converse part actually shows a stronger result.

\begin{theorem}\label{cor:1}
Let $f:\H \to \cp$ be a holomorphic function such that properties $({\rm II})$ and $({\rm IV})$ in Theorem \ref{thm:maineven} hold true, but also the following weaker version of property $({\rm III})$: 
\begin{enumerate}
\item[(III$^*$)] There is $c_2\geq c_1$, and for every $M>0$ there is $B_M>0$, such that for any trigonometric polynomial $p(x)=\sum_{n=1}^N p_n e^{2\pi i \theta_n x}$, with $\{\theta_n\} \subset [0,M]$, we have
$$
\limsup_{T\to\infty} \bigg|\frac{1}{2T}\int_{-T}^T f(x+ic_2)\ov{p(x)}\d x\bigg| \leq B_M \max |p_n|;
$$
\end{enumerate}
Then there is $c\geq 0$ such that the limit
$$
\E f(\la) := \lim_{T\to\infty} \frac{1}{2T}\int_{-T+iy}^{T+iy} f(z)e^{-2\pi i \la z}\d z
$$
exists for every $\la\in\r$ and $y>c$, is independent of $y$, vanishes identically for $\la<0$ and defines a locally summable function. Moreover, there is a unique real-valued locally finite measure $\mu$ of degree at most $2$ satisfying \eqref{eq:repkerLphiid}, and if we define $\a$ as in \eqref{eq:adef} then $(\mu,a)$ is an real-antipodal $\fs$-pair and identity \eqref{eq:fdef} holds.
\end{theorem}

Note that condition (III$^*$) is equivalent to $\E f(\la)$ being locally summable only, no growth required.

\begin{remark}\label{rem:exptodeg2}
We claim that if $(\mu,a)$ is an $\fs$-pair such such that $\mu\geq 0$ and $\a$ has exponential growth then $\deg(\mu)\leq 2$. Indeed, let $\p\in C^\infty_c(\r)$ be even, $0\leq \p \leq 1$, $\int \p =\p(0)=1$ and $\supp(\p) \subset (-1,1)$. Let $\p_\ep(x)=\p(x/\ep)/\ep$ and $\psi_\ep(x)=\ft \p(x/\ep)/\ep$ for $0<\ep <1$. Let $g(x)=e^{-2 \pi C |x|}$, with $C>0$ sufficiently large, and note $\ft g(\xi) = \frac{C}{\pi(C^2+\xi^2)}$. We have
$$
\int_\r (\ft g \ft \p_\ep)*\psi_\ep(t-t_0)\d \mu(t) = \sum_{\la \in \r} a(\la)g * \p_\ep(\la)\ft \psi_\ep(\la)e^{-2\pi i \la t_0 },
$$
for any $t_0\in \r$. It is not hard to show that $|g * \p_\ep(\la)\ft \psi_\ep(\la)| \leq e^{-2 \pi C (|\la|-1)}$. Thus, we conclude the right hand side above is uniformly bounded in absolute valued for $0<\ep<1$. Since $\mu\geq 0$ and $(\ft g \ft \p_\ep)*\psi_\ep \to \ft g$ pointwise, as $\ep\to 0$, we can apply Fatou's Lemma to deduce $\int \frac{\d \mu(t)}{C^2+(t-t_0)^2}<C_0$, for some constant $C_0>0$ independent of $t_0$. In particular we obtain that\footnote{Hence $\mu$ is translation-bounded.}
$$
\sup_{t_0\in \r} \mu([t_0-1,t_0+1])\leq (C^2+1)C_0.
$$
\end{remark}


\section{Applications} We now apply Theorems \ref{thm:maineven} and \ref{cor:1} to classify crystalline pairs. Some of the results and proofs presented here would be better understood after reading Sections \ref{sec:AP}, \ref{sec:apc} and \ref{sec:rep}. We say an entire function $E$ is of \emph{Hermite-Biehler} class if 
\begin{align}\label{def:hermitebiehler}
|E^*(z)| < |E(z)| \quad \text{for all } \  z\in \cp^+,
\end{align}
where $E^*(z)=\ov{E(\ov z)}$. In this case we always write $E=A-iB$, where $A$ and $B$ are real entire functions\footnote{Entire functions which attain only real values on the real line.} defined by $A=(E^*+E)/2$ and $B=(E^*-E)/(2i)$, and we always write $\Theta=E^*/E$. In particular we have the following identity
\begin{align}\label{eq:AbThetaid}
i\frac{A}{B} = \frac{1+\Theta}{1-\Theta}.
\end{align}
Note that both $A$ and $B$ have only real zeros. We denote by $\varphi=\p_E$ a phase function associated to $E$. This is characterized by the condition $e^{i\varphi(x)} E(x) \in\r$ for all real $x$ ($\p$ is unique modulo an integer multiple of $\pi$). For instance, one could take $\p=\frac{1}{2i}\log \Theta$ after branch cutting all zeros and poles of $\Theta$ by vertical semi-lines. It is not hard to show that
$$
\varphi'(x)=\Re i\frac{E'(x)}{E(x)} = \partial_y \log \Theta(x+iy)|_{y=0} > 0,
$$
for all real $x$, and so $\p(x)$ is an increasing function of real $x$. We also have that
$$
e^{2i\p(x)} = \frac{A(x)^2}{|E(x)|^2}  - \frac{B(x)^2}{|E(x)|^2} + 2i\frac{A(x)B(x)}{|E(x)|^2},
$$
for all real $x$. As a consequence, the points $\ga\in\r$ such that $\p(\ga)\equiv 0 \,(\text{mod }  \pi)$ coincide with the real zeros of $B/E$ and the points $s\in\r$ such that $\p(s)\equiv \pi/2\,(\text{mod }  \pi)$ coincide with the real zeros of $A/E$ and, because $\p'>0$, these zeros are simple. In particular, $A/B$ has only simple real zeros and simple real poles, which interlace, and
 $$
\{\ga\in \r : \, \p_j(\ga)\equiv 0 \,({\rm mod} \, \pi)\} = {\rm Zeros}(B_j/A_j) \quad \text{and} \quad  \frac1{\p_j'(\ga)} = \res{\ga}(A_j/B_j)>0.
 $$
Moreover, if $E$ has no real zeros then $A$ and $B$ have only real simple zeros which interlace\footnote{For more on this we recommend \cite{dB} and the introduction of \cite{GL}.}. The next result characterizes $\fs$-pairs $(\mu,a)$ for which $\mu\geq 0$ has locally finite support.

\begin{theorem}\label{thm:HBmupos}
Let $E=A-iB$ be of Hermite-Biehler class such that $A/B \in \apc$ and $\la\in \r \mapsto \E (A/B)(\la)$ is locally summable. Let
\begin{align}
\begin{split}\label{def:amu}
a(\la)& =\ov{a(-\la)}=\E(iA/B)(\la)=\lim_{T\to\infty} \frac{1}{T}\int_{-T + i}^{T+i} i\frac{A(z)}{B(z)} e^{-2\pi i \la z} \d z \quad (\text{for } \la>0), \\
\frac12 a(0)&=\re \E (iA/B)(0) = \Re \lim_{T\to\infty} \frac{1}{T}\int_{-T + i}^{T+i}  i\frac{A(z)}{B(z)} \d z,  \\
\mu & = 2\pi \sum_{\, \p_E(\ga)\equiv 0 \,({\rm mod} \, \pi)} \frac{1}{\p'_E(\ga)} {\del}_{\ga}.
\end{split}
\end{align}
Then $(\mu,a)$ is a real-antipodal $\fs$-pair such that $\mu\geq 0$ and $\mu$ has locally finite support\footnote{A set $S\subset \r$ is locally finite if $\# (S\cap (a,b)) < \infty$ for any $a<b$.}. Conversely, if $(\mu,a)$ is a real-antipodal $\fs$-pair such that $\mu\geq 0$, $\supp(\mu)$ is locally finite and $\a$ has exponential growth, then $(\mu,a)$ has to be built from the construction above.
\end{theorem}

\begin{remark}\label{rem:ABTheta}
Observe that by Lemma \ref{lem:fdisc}(iii),(iv),(v) and identity \eqref{eq:AbThetaid}, $A/B\in \apc$ if and only if $\Theta=E^*/E\in \apc$ and $A/B$ has locally finite spectrum if and only if $\Theta$ has. Note also that if $E,E^*\in \apc$ then $\Theta\in \apc$. In particular, if $E$ is a {trigonometric polynomia}l of Hermite-Biehler class, then $iA/B$ belongs to $\apc$ and has locally finite spectrum, hence Theorem \ref{thm:HBmupos} applies and the pair $(\mu,a)$ is crystalline.
\end{remark}

\begin{remark}[The  Kurasov \& Sarnak construction] \label{rem:KS}
Theorem \ref{thm:HBmupos} generalizes a construction of Kurasov \& Sarnak \cite{KS}. They construct a crystalline $\fs$-pair $(\mu,a)$ by letting
$$
\mu = \sum_{Q(\ga)=0} m(\ga)\del_\ga,
$$
where $Q$ is any given trigonometric polynomial with only real zeros $\ga$ and multiplicities $m(\ga)$. Theorem \ref{thm:HBmupos} majorates this construction. Indeed, since $Q$ has finite exponential type, Hadamard's factorization implies that 
$$
Q(z)=z^n e^{c_1z+c_2}\prod_{\ga\neq 0} (1-z/\ga)e^{z/\ga}
$$
for some $c_1,c_2\in\cp$ and $n\geq 0$, where the $\ga's$ are the real zeros of $Q$. In particular $e^{-c_2-iz\Im c_1 }Q$ is a real entire trigonometric polynomial, so we can assume that $Q$ is real on the real line. If that is the case then $c_1\in \r$ and
$$
\Re i\frac{Q'(z)}{Q(z)} =\sum_{\ga} \frac{y}{(x-\ga)^2+y^2}>0.
$$
We obtain that $E=Q'-iQ$ is a trigonometric polynomial of Hermite-Biehler class by a routine calculation. By the previous remark, Theorem \ref{thm:HBmupos} applies. It is also easy to see that the zeros of $Q$ coincide with the points $\ga\in\r$ such that $\p_E(\ga)\equiv 0 \,({\rm mod} \, \pi)$ and $\p'_E(\ga)^{-1}= \res{\ga}(A/B)=m(\ga)$. Thus, by Theorem \ref{thm:HBmupos} (with $\a$ as in \eqref{def:amu}), both $\mu$ and $\a$ have locally finite support, and so $(\mu,a)$  is crystalline. A simple example is given by the Hermite-Biehler trigonometric polynomial
$$
E(z)=\tfrac{1}2({(\pi +1)}e^{-\pi i z}+ e^{iz} + {(\pi -1)}e^{\pi i z}) = \pi \cos(\pi z) + \tfrac12 \cos(z) -i(\sin(\pi z)+\tfrac12\sin(z)),
$$
which produces the crystalline measure $\mu=\sum_{\sin(\pi \ga)+\tfrac12\sin(\ga)=0} \del_\ga$.
\end{remark}

\begin{remark}[The Ulanovskii \& Oleveskii result] \label{rem:OU}
We now explain how one can prove the main result of Ulanovskii \& Oleveskii \cite{OU} using Theorem \ref{thm:HBmupos}, but under milder decay conditions. We claim that if $(\mu,a)$ is an $\fs$-pair such $\mu\geq 0$ is $\n$-valued, $\a$ has locally finite support and exponential growth then the Hermite-Biehler function $E=A-iB$ given by Theorem \ref{thm:HBmupos} is a trigonometric polynomial, and so the construction from Remark \ref{rem:KS} applies. Indeed, since $\deg(\mu)\leq 2$ one can take $\B$ to have the same zeros of $B$, but with multiplicities given by the weights of $\mu$. We can take such $\B$ with $\text{order}(\B)\leq 1$ by \cite[Thm. 6, p. 16]{L}. Using \eqref{eq:repkerLphiid} we conclude that $\re i\B'/\B=\re A/B$ in $\H$, and so Poisson representation shows that $A/B = \B'/\B + h$ for some $h\in \r$, hence we can assume that $A=B'$. Since $\a$ has locally finite support, then $B'/B$ has locally finite spectrum contained in $[0,\infty)$. Lemma \ref{lem:expprim} shows that $B\in \apc$ and it has locally finite spectrum bounded from below. The class $\apc$ is closed under differentiation, and so $B'\in \apc$, and therefore, $E\in \apc$, both with locally finite spectrum bounded from below. Note that $B'$ also has order $\leq 1$, and so $E$ has order $\leq 1$. We can then apply Lemma \ref{lem:order1apc} to conclude that $E$ is a trigonometric polynomial.
\end{remark}

\begin{proof}[Proof of Theorem \ref{thm:HBmupos}]
Let $f=iA/B$. Since $E=A-iB$ is Hermite-Biehler, we obtain that $\re f >0$ in $\H$, and so $|e^{-f}|<1$ in $\H$, hence $e^{f}$ is of bounded type. Moreover, $\E f(\la)$ is locally summable by hypothesis. We conclude that Theorem \ref{cor:1} applies. Poisson representation \cite[Thm. 4 \& Prob. 89]{dB} for $f$ implies that $\mu$ has the form given by \eqref{def:amu}. Conversely, if $(\mu,a)$ is a real-antipodal $\fs$-pair such that $\mu\geq 0$, $\supp(\mu)$ is locally finite and $\a$ has exponential growth, then we can apply Theorem \ref{thm:HBmudisc} to obtain that \eqref{def:amu} holds for some Hermite-Biehler function $E=A-iB$ such that $A(z+ic)/B(z+ic)$ belongs to $\apc$ for some $c>0$. However, Lemma \ref{lem:fdisc} also shows that $\Theta(z+ic)=E^*(z+ic)/E(z+ic)$ belongs to $\apc$, and since $|\Theta|<1$ in $\H$, Lemma \ref{lem:altdefapc} shows that $\Theta\in \apc$ and, by Lemma \ref{lem:fdisc} again we obtain that $A/B\in \apc$.  Since $a(\la)$ is locally summable and $a(\la)=\E f(\la)$ for $\la>0$, then $\E f(\la)$ is locally summable.
\end{proof}

The next result classifies $\fs$-pairs $(\mu,a)$ for which $\mu$ has locally finite support.

\begin{theorem}\label{thm:HBmudisc}
Let $(\mu,a)$ be an $\fs$-pair such that $\mu$ has locally finite support, $\deg(\mu)\leq 2$ and $\a$ has exponential growth. Then there exists four Hermite-Biehler functions $E_j=A_j-iB_j$ and numbers $p_j\in \{0,1\}$, for $j=1,2,3,4$, such that if we let $f=ip_1A_1/B_1-ip_2A_2/B_2$ and $g=ip_3A_3/B_3-ip_4A_4/B_4$ then $f(\cdot+ic)$ and $g(\cdot+ic)$ belong to $\apc$ for some $c>0$, $\E f(\la)$ and $\E g(\la)$ are locally summable, and:
\begin{align}
\begin{split}\label{def:genmua}
        \frac1{2\pi}{\mu} & = p_1\sum_{\, \p_1(\ga)\equiv 0 \,({\rm mod} \, \pi)} \frac{1}{\p_1'(\ga)} {\del}_{\ga}-p_2\sum_{\, \p_2(\ga)\equiv 0 \,({\rm mod} \, \pi)} \frac{1}{\p_2'(\ga)} {\del}_{\ga} \\
 & \ \  -  ip_3\sum_{\, \p_3(\ga)\equiv 0 \,({\rm mod} \, \pi)} \frac{1}{\p_3'(\ga)} {\del}_{\ga}+ip_4\sum_{\, \p_4(\ga)\equiv 0 \,({\rm mod} \, \pi)} \frac{1}{\p_4'(\ga)} {\del}_{\ga}; \\
a(\la) & =\E f(\la)-i\E g(\la) \quad \text{and} \quad a(-\la)  =\ov{\E f(\la)+i\E g(\la)} \quad (\text{for } \la>0);\\
\tfrac1{2}a(0)  & =\re \E f(0)-i\re\E g(0),
\end{split}
\end{align}
where $\p_j$ is the phase function associated with $E_j$. Moreover: (i) $\mu\geq 0$ if and only if $p_1=1$ and $p_2=p_3=p_4=0$; (ii) $\a$ has locally finite support if and only if $f$ and $g$ have locally finite spectrum. Furthermore, if $\mu$ is $\n$-valued then we can take $A_1=B_1'$.
\end{theorem}

\begin{proof}[Proof of Theorem \ref{thm:HBmudisc}]
By real-antipodal splitting we can assume that $\mu$ is real-valued (so $p_3=p_4=0$). We can then apply Theorem \ref{thm:maineven} and obtain that
$$
f(z) = \frac{a(0)}{2} + \lim_{T\to\infty} \sum_{0<\la < T} (1-\la/T)a(\la)e^{2\pi i \la z} = ih + \frac{1}{2\pi i} \int_\r\frac{1+tz}{t-z}\frac{\d\mu(t)}{1+t^2},
$$
for some $h\in \r$, where $f$ is holomorphic in $\H$ and $f(\cdot + ic)\in \apc$ for some $c>0$. 
Let $\mu=\sum_{\ga \in \Lambda} r(\ga) \del_\ga$ where $\Lambda \subset \r$ is locally finite and $r(\ga)=r_1(\ga)-r_2(\ga)$, with $r_j\geq 0$. Let 
$$
f_j(z) =  \frac{1}{2\pi i}\int_\r \frac{1+tz}{t-z}\frac{\d \mu_j(t)}{1+t^2}
$$
for $j=1,2$, where $\mu_j=\sum_{\ga \in \Lambda} r_j(\ga) \del_\ga$. Observe that each $f_j$ is meromorphic in $\cp$ with only simple poles (possibly) at $\Lambda$. Since $\mu_j\geq 0$ we have $if_j(\H) \subset \H\cup\{0\}$, and we can then apply a classical result \cite[p.  308, Thm. 1]{L} to deduce this happens exactly when there are two real entire functions $A_j$ and $B_j$ with only real zeros that interlace and such that $if_j=-p_jA_j/B_j$, with $p_j\in \{0,1\}$ to account for the case when $\mu_j=0$ for some $j=0,1$. Since $\re iA_j/B_j>0$, we deduce that $E_j=A_j-iB_j$ is of Hermite-Biehler class, and so, Poisson representation shows that
$$
\mu = 2\pi p_1\sum_{\, \p_1(\ga)\equiv 0 \,({\rm mod} \, \pi)} \frac{1}{\p_1'(\ga)} {\del}_{\ga}-2\pi p_2\sum_{\, \p_2(\ga)\equiv 0 \,({\rm mod} \, \pi)} \frac{1}{\p_2'(\ga)} {\del}_{\ga}.
$$
Noting that $f=f_1-f_2+ih$, we obtain $f_1(\cdot+ic)-f_2(\cdot+ic)\in \apc$, and since $\a$ has exponential growth we conclude that $a(\la)=\E (f_1-f_2)(\la)$ for $\la>0$ and $\frac12 a(0)=\re \E (f_1-f_2)(0)$. If $\mu\geq 0$ then $(p_1,p_2)=(1,0)$. Also, $\a$ has locally finite support if and only if $f_1-f_2$ has locally finite spectrum. Finally, if $\mu$ is $\n$ valued then, since $\deg(\mu)\leq 2$, one might take $\B$ to have zeros exactly at support of $\mu$ with multiplicity determined by the weights of $\mu$ and order at most $1$. In particular, Hadamard factorization and log-differentiation show that $\re i\B'/\B=\re i A_1/B_1$ and so $f_1=ih + i\B'/\B$, and $\mathcal{E}=\B'-i\B$ is now Hermite-Biehler.
\end{proof}

\begin{remark}\label{rem:genamu}
We note that it follows from Theorems \ref{thm:maineven} and \ref{cor:1} that Theorem \ref{thm:HBmudisc} is sharp. That is, if we are given four Hermite-Biehler functions $E_j=A_j-iB_j$ and numbers $p_j\in \{0,1\}$, for $j=1,2,3,4$, such that $f=ip_1A_1/B_1-ip_2A_2/B_2$ and $g=ip_3A_3/B_3-ip_4A_4/B_4$ belong to $\text{AP}(\cp+ic)$ for some $c>0$, then the pair $(\mu,a)$ defined by \eqref{def:genmua} is an $\fs$-pair. A crucial instantiation of Theorem \ref{thm:HBmudisc} is the difference of two prime summation formula of Guinand, explained of Section \ref{sec:ACM}. It can be obtained from our results in the following way. Let $\chi_1$ and $\chi_2$ be two primitive and even Dirichlet characters of distinct modulus $N_1$ and $N_2$, and let 
$$
L(\chi_j,s) = \sum_{n\geq 1} \chi_j(n)n^{-s}
$$
be their associated L-functions for $j=1,2$ and $\re s>1$. One can attach to each of them the function
$$
B_j(z)=\ft \chi_j(1)^{-1/2}(N_j/\pi)^{-iz/2}\Gamma(1/4-iz/2)L(\chi_j,1/2-iz)
$$
for complex $z$, where $\ft \chi_j(1)=N_j^{-1/2}\sum_{n=1}^N \chi(j)e^{2\pi i n/N} \in \{\pm 1, \pm 1\}$. Thanks to their functional equations, it can be shown that $B_j$ is real on the real line, this is, $\ov{B_j(\ov z)}=B_j(z)$, and that each $B_j$ is an entire function of order $\leq 1$. Assuming the Riemann Hypothesis for both $L$-functions, each $B_j$ have only real zeros and those coincide (with multiplicity) with the ordinates of the nontrivial zeros of each $L$-function. We now let $E_j=A_j-iB_j$, with $A_j=B'_j$. Hadamard's factorization shows that $ \re iA_j/B_j \geq 0$ in $\cp^+$ and so each $E_j$ is a Hermite-Biehler function. Now observe that
\begin{align*}
i\frac{A_1}{B_1} - i\frac{A_2}{B_2} & = \frac12\log(N_1/N_2) + \frac{L'(\chi_1,1/2-iz)}{L(\chi_1,1/2-iz)} - \frac{L'(\chi_2,1/2-iz)}{L(\chi_2,1/2-iz)}  \\ & = \frac12\log(N_1/N_2) + \sum_{n \geq 2} \frac{\Lambda(n)(\chi_1(n)-\chi_2(n))}{\sqrt{n}}e^{2\pi i \tfrac{\log(n)}{2\pi} z}
\end{align*}
when $\im z>1/2$, where $\Lambda(n)$ is the von Mongoldt function.
It follows that $f=i\frac{A_1}{B_1} - i\frac{A_2}{B_2}$ is almost periodic in $\cp^++i/2$. The previous results show that the pair
\begin{align*}
\mu & = 2\pi\bigg(\sum_{L(\chi_1,1/2+i\ga_1)=0} \del_{\ga_1}- \sum_{L(\chi_2,1/2+i\ga_2)=0} \del_{\ga_2}\bigg)
\\
a &= \log(N_1/N_2)\1_0 + \sum_{n\geq 2} \frac{\Lambda(n)}{\sqrt{n}}\bigg[(\chi_1(n)-\chi_2(n))\1_{\tfrac{\log(n)}{2\pi}} + (\ov{\chi_1(n)}-\ov{\chi_2(n)})\1_{-\tfrac{\log(n)}{2\pi}}\bigg]
\end{align*}
is a $\fs$-pair (above, $\1_c(\la)=1$ if $\la=c$ and $\1_c(\la)=0$ otherwise).
\end{remark}

In what follows we say that an entire function $E$ is of \emph{P\'olya class} if $E$ has no zeros in $\H$, $|E^*/E|\leq 1$ in $\H$ and $y\mapsto |E(x+iy)|$ is nondecreasing for $y>0$. A classical result \cite[Thm. 7]{dB} states that a function $E$ is of P\'olya class if and only if
\begin{align}\label{eq:poylafact}
E(z) = E^{(r)}(0)(z^r/r!)e^{-dz^2-ibz} \prod_{n} (1-z/\ov{z_n})e^{z\re 1/z_n},
\end{align}
where $d\geq 0$, $\re b\geq 0$, $\{\ov z_n=x_n-iy_n\}$ are the nonzero zeros of $E$ (with $y_n\geq 0$) and
$$
\sum_{n} \frac{1+y_n}{x_n^2+y_n^2}<\infty.
$$ 
The P\'olya class can be also defined as functions that are uniform limits in compact sets of polynomials with no zeros in $\H$ (see \cite[Problem 12]{dB}). Moreover, any function of P\'olya class has order at most $2$ (\cite[Problem 10]{dB}).

The next result is a brand new classification theorem\footnote{In a way, this is the most technically hard result of the paper.} for nonnegative crystalline measures. It generalizes a result of Olevekii \& Ulanovskii \cite{OU} in the uniformly discrete case, dropping the $\n$-valuedness assumption.

\begin{theorem}\label{thm:5}
Let $(\mu,a)$ be an $\fs$-pair. Assume that:
\begin{enumerate}
\item $\mu$ has uniformly discrete support and there is $\delta>0$ such that $\mu(\{x\})\geq \delta$ for all $x\in \supp(\mu)$;
\item There is $b>0$ such that $\supp(a)\cap (0,b)=\emptyset$ and $\a$ has exponential growth.
\end{enumerate}
Then there exists an entire function $E=A-iB$ of Hermite-Biehler class and of finite exponential type such that $E,E^*\in \apc$ (and so $A/B\in \apc$), $\spec(E)$ is bounded and \eqref{def:amu} holds. If in addition $\supp(a)$ is locally finite, then $E$ is a trigonometric polynomial and Remark \ref{rem:KS} applies.
\end{theorem}

\begin{proof}
\underline{Step 1}. We first apply Theorem \ref{thm:HBmupos} to conclude there is an Hermite-Biehler function $E=A-iB$ such that $A/B\in \apc$ and \eqref{def:amu} holds.  We can assume that $E$ has no real zeros as these can be removed (with a Weierstrass product) without altering $A/B$ or the fact that $E$ is Hermite-Biehler. In particular, $\supp(\mu)=\{\ga\in \r : \p(\ga)\equiv 0 \,({\rm mod} \, \pi)\}=\text{Zeros}(B)$, where $\p$ is the phase function associated with $E$, and these zeros are all simple. The zeros of $A$ are also simple and interlace  those of $B$. Since $\mu$ is bounded from below on its support and $\a$ has exponential growth,  we can apply Remark \ref{rem:exptodeg2} to obtain $$c<\p'(\ga)<1/c,$$ for some $0<c<1$, whenever $B(\ga)=0$. In particular, we have $\sum_{B(\ga)=0} \frac{1}{1+\ga^{2}}<\infty$ and, since the zeros of $A$ interlace those of $B$, we also have $\sum_{A(s)=0} \frac{1}{1+s^{2}}<\infty$.  We claim that we can assume $E$ is has order at most $1$. Indeed, define the canonical products
$$
\B(z) = b z^p \prod_{B(\ga) = 0,\ \ga\neq 0} (1-z/\ga)e^{z/\ga} \quad \text{and} \quad \A(z) = z^q \prod_{A(s) = 0,\ s\neq 0} (1-z/s)e^{z/s},
$$
for some $p,q\in \{0,1\}$, where $b\in\r$ is chosen so that $(A(z)/B(z))/(\A(z)/\B(z))=1+O(z)$ for $z\to 0$. Hence both $\A$ and $\B$ have order at most $1$ by a classical result of Borel \cite[Thm. 6, p. 16]{L}. Since $\re iA/B \geq 0$ in $\H$ and, a straitfoward computation, shows that $\re i\A/\B\geq 0$ in $\H$, we conclude that both $A/B$ and $\A/\B$ are of bounded type and so $F= (A/B)/(\A/\B)$ also is and has no zeros. Since $F=F^*$, a classical result of Krein \cite{Kr} (see also \cite[Thm. 1, p. 115]{L2}) shows that $F$ is of exponential type, and so $F(z)=e^{hz}$, for some $h\in \r$. Since $F$ is of bounded type we have $h=0$. We can then define $\widetilde{E} =\A-i\B$ to finish the claim.

\underline{Step 2}. Let $\{\ov z_n=x_n-iy_n\}$ be the zeros of $E$ with $y_n> 0$. We claim $\sup_{n} y_n <\infty$ and that $E$ is of P\'oyla class. Indeed, since $f=iA/B\in \apc$, by Remark \ref{rem:ABTheta} we have $\Theta=E^*/E=(1-f)/(1+f) \in \apc$. Since $\spec(f)\subset [0,\infty)$, a simple computation using Lemma \ref{lem:fdisc}(i) shows
$\E \Theta(0)=\frac{1-\E f(0)}{1+\E f(0)}$. However, $2\re \E f(0) = a(0)>0$ (because $\mu\geq 0$), thus $|\E \Theta(0)|<1$. Assume now $\E \Theta(0)\neq 0$. Then Lemma \ref{lem:fdisc}(i) shows straightforwardly that $\Theta(x+iy)$ has no zeros for large $y>0$. Assume otherwise that $\E \Theta(0)=0$, that is, $\E f(0)=1$. Since $a(\la)=\E f(\la)$ for $\la>0$ we deduce that $\spec(f) \subset \{0\}\cup [b,\infty)$ where $b=\inf\{\la>0 : a(\la)\neq 0\}>0$ by assumption. Since  $f$ has no zeros in $\H$, we can then apply Lemma \ref{lem:fdisc}(ii) to conclude that $b\in \spec(f)$. Writing $f=1-2g$, with $\spec(g)\subset [b,\infty)$ and $\E g(b)\neq 0$, we can apply Lemma \ref{lem:fdisc}(i) to conclude that $|g(x+iy)|<1$ for large $y>0$, and so
$$
\Theta(z) = \frac{g(z)}{1-g(z)} =  \sum_{n\geq 1} g(z)^{n},
$$
hence $\spec(\Theta)\subset [b,\infty)$ and $\E \Theta(b)\neq 0$. We can then apply again Lemma \ref{lem:fdisc}(i) to conclude that $\lim_{y\to\infty}\sup_{x}|e^{-2\pi i b (x+iy)}\Theta(x+iy)- \E \Theta(b)|=0$, hence $\Theta(z)$ has no zeros for large $\im z$. Finally, since $E$ has order at most $1$ the following sum is now finite
$$
\sum_{n} \frac{1+y_n}{x_n^2+y_n^2}<\infty.
$$ 
and $E$ has the following factorization
 $$
E(z) = E(0)e^{-ihz} \prod_{n} (1-z/\ov{z_n})e^{z\re 1/z_n}
$$
for some $h\geq 0$. Hence $E$ is of P\'oyla class.

\underline{Step 3}. We claim now that $\sup_{x\in \r} \p'(x)<\infty$. Since \eqref{eq:repkerLphiid} holds true for $f=iA/B$, we obtain that
$$
K(w,z):=\frac{B(z)\ov{A(w)}-\ov{B(w)}{A(z)}}{\pi(z-\w)} = \sum_{B(\ga)=0} \frac{1}{\pi\p'(\ga)} \frac{B(z)\ov{B(w)}}{(\ga-z)(\ga-\w)}
$$
for all $z,w\in\cp$.  A routine computation shows that
$$
\p'(x) = \re i \frac{E'(x)}{E(x)} =  \frac{\pi K(x,x)}{|E(x)|^2} = \sum_{B(\ga)=0} \frac{\sin^2 \p(x)}{\p'(\ga)(\ga-x)^2}.
$$
Enumerate the $\text{Zeros}(B)=\{\ga_n\}$ and let $c>0$ be so small that $\ga_{n+1}-\ga_n\geq c$ for all $n$. Let $0<\ep<c/10$ to be chosen later.  Observe first that if ${\rm dist}(x,\{\ga_n\})=|x-\ga|\geq \ep$ (with $\ga=\ga_{n_0}$) then $|x-\ga_{n_0+k}| \geq \ep+|k|c/2$, and so
$$
\p'(x) \leq 2\sum_{k \geq 0} \frac{1}{c(\ep+kc/2)^2} = O(\ep^{-2}).
$$
On the other hand, if ${\rm dist}(x,\{\ga_n\})=|x-\ga|\leq \ep$, then $|\ga_{n_0+k}-x|\geq \frac{c}2|k|$, for $k\neq 0$, and 
$$
\p'(\ga)\p'(x) \leq  \frac{\sin^2 \p(x)}{(\ga-x)^2} + C\sin^2 \p(x)
 $$
 for some $C>0$, independent of $\ep$ and $\ga$. Using the inequality $x^2\geq \sin^2 x$ for all real $x$, and factoring $\sin^2 \p(x)$, we obtain
 $$
\frac{\p'(\ga)\p'(x)}{\sin^2 \p(x)} \leq \frac{ (\pi/c)^2}{\sin^2[(\ga-x)\pi/c]} + C
 $$
 We conclude that the function
 $$
\xi(x)= \p'(\ga)\cot \p(x) - \frac{\pi}{c}\cot[\pi(x-\ga)/c] + C(x-\ga),
 $$
 which is analytic in a complex neighbourhood of the segment $I=[\ga-\ep,\ga+\ep]$, is nondecreasing for $x\in I$ and $\xi(\ga)=0$. We conclude that 
 $$
 (x-\ga)\p'(\ga)\cot \p(x)  \geq \frac{\pi(x-\ga)}{c}\cot[\pi(x-\ga)/c] - C(x-\ga)^2
 $$
 for $x\in I$. Since the function $ \frac{\pi(x-\ga)}{c}\cot[\pi(x-\ga)/c]$ has value $1$ for $x=\ga$, there is $\ep_0=\ep_0(c,C)>0$, with $0<\ep_0<c/10$, such that the right hand side above is bounded from below by $ \frac{\pi(x-\ga)}{2c}\cot[\pi(x-\ga)/c]$
for $x\in I_0:=(\ga-\ep_0,\ga+\ep_0)$, and so
$$
 (x-\ga)\p'(\ga)\cot \p(x)  \geq \frac{\pi(x-\ga)}{2c}\cot[\pi(x-\ga)/c]
 $$
 for $x\in I_0$. We now select $\ep=\ep_0$. Observing the right hand side above is positive for $x\in I_0$, we can square both sides above, factor out $(x-\ga)^2$, lower bound $\p'(\ga)\geq c$ and isolate $\sin^2 \p $ to obtain
$$
\sin^2 \p(x) \leq \frac{1}{1+C\cot^2[\pi(x-\ga)/c]}
$$
for $C=\pi^2/(4c^4)$. Note also that for $x\in I_0$ we have $|x-\ga_{n_0+k}|\geq |x-\ga_{n_0}-kc|$ for any $k\in \z$, hence
\begin{align*}
\p'(x) \leq \frac{1}{c}\sum_{k\in \z}\frac{\sin^2 \p(x)}{(x-\ga+kc)^2} & =\frac{\pi^2}{c^3} \frac{\sin^2 \p(x)}{\sin^2[\pi(x-\ga)/c]} \\ &\leq  \frac{\pi^2}{c^3}\frac{1}{\sin^2[\pi(x-\ga)/c]+C\cos^2[\pi(x-\ga)/c]}  \leq \frac{\pi^2}{c^3}.
\end{align*}
for $x\in I_0$. This proves the claim.

\underline{Step 4}. We claim the zeros of $E$ are separated from the real axis, that is, there is $c>0$ such that $y_n\geq c$ for all $n$. Indeed, observe that  since $E(z)e^{i\p(z)}=E^*(z)e^{-i\p(z)}$ whenever $\p(z)$ is defined (any neighborhood of $\r$ not containing $\{z_n\}_{n\in\z}\cup \{\ov{z_n}\}_{n\in \z}$) we obtain that
$$
\p'(z) = \partial_z \frac{1}{2i} \log \Theta(z) =  \frac{1}{2i} \frac{\Theta'(z)}{\Theta(z)},
$$
with $\Theta=E^*/E$. Using the factorization of $E$ we obtain
$$
\Theta(z) = e^{2i h z}\prod_{n\in \z} \frac{1-z/z_n}{1-z/\ov{z_n}},
$$
and so
$$
\p'(x) = h +  \sum_{n\in \z} \frac{y_n}{|x-z_n|^2}
$$
for real $x$. Hence $\p'(x_n) \geq 1/y_n$, and since $\p'(x)$ is bounded, the claim follows.

\underline{Step 5}. We claim that $\p'\in \apr$ and $\spec(\p')\subset [b,\infty)$. Let $c>0$ be small enough such that $\Theta$ has no zeros or poles in the strip $S=\{-c<\im z <c\}$.  First we show that $\Theta$ is bounded in every horizontal strip contained  in $S$. We only need to show this for $z\in S$ with $\im z\leq 0$. Indeed, since
$$
\frac12 \partial_y \log |\Theta^*(x+iy)|  =   h + \sum_{n\in\z}\frac{y_n[(x-x_n)^2+y_n^2-y^2]}{|z-z_n|^2|z-\ov{z_n}|^2}
$$
and $y_n^2-y^2\leq 3(y_n-y)^2$ for $0\leq y\leq c$, we deduce that
$$
\frac12 \partial_y \log |\Theta^*(x+iy)| \leq h + 3\sum_{n\in \z} \frac{y_n}{|x-z_n|^2} \leq 3\p'(x) \leq C
$$
for some $C>0$.  Since $\log |\Theta^*(x)|=0$, integration shows that $\Theta^*(x+iy)\leq e^{2Cy}$ for $0\leq y\leq c$. Since $\Theta\in \apc$, we can now apply Lemma \ref{lem:altdefapc} to deduce that $\Theta(\cdot-ic)\in \apc$ and $\Theta'(\cdot-ic)\in \apc$. By Lemma \ref{lem:fdisc}(iv), we get that $|\Theta|$ is bounded away from zero in every horizontal strip contained in $S$. We deduce that $\Theta'/\Theta\in \apr$, and so\footnote{If we let ${\rm AP}(S)$ be defined in the same way as in $\apc$, but considering only horizontal strips strictly contained in $S$, we have showed that $\p'\in {\rm AP}(S)$.} $\p'\in \apr$. Step 1 shows that $\spec(\Theta) \subset \{0\}\cup [b,\infty)$ and $\E \Theta(b)\neq 0$, hence $\spec(\Theta'/\Theta)\subset [b,\infty)$ and $\spec(\p')\subset [b,\infty)$.

\underline{Step 6}. Since $E$ is of P\'olya class we have $\re iE'/E\geq 0$ in $\H$, and Poisson representation guarantees that
$$
\frac{iE'(z)}{E(z)} =  id - ipz + \frac{1}{\pi i} \int_\r \frac{1+tz}{t-z}\frac{\p'(t)\d t}{1+t^2},
$$
for some $d\in \r$ and $p\geq 0$, where  $\deg(\p'(t)\d t) \leq 2$. This fact, in conjunction with the factorization of $E$ shows that
$$
\re \frac{iE'(z)}{E(z)} = py + \frac{y}{\pi} \int_\r \frac{\p'(t)\d t}{(x-t)^2+y^2} = h + \sum_{n}\frac{y+y_n}{(x-x_n)^2+(y+y_n)^2}.
$$
Since $\sum_{n} \frac{1+y_n}{x_n^2+y_n^2}<\infty$, we conclude that $p=\lim_{y\to\infty} \re \frac{iE'(iy)}{yE(iy)}=0$. We can apply Lemma \ref{lem:aprpoisson} so deduce that $iE'/E\in \apc$ and $\spec(iE'/E)\subset \{0\}\cup [b,\infty) $. We can apply Lemma \ref{lem:expprim} to conclude that $E\in \apc$ and that $\spec(E)$ is bounded from below. Since $E^*=\Theta E$, we conclude that $E^*\in \apc$ also and that $\spec(E^*)$ is bounded from below. Finally, we can use Lemma \ref{lem:order1apc} to deduce the spectrum of $E$ (and of $E^*$) is bounded and $E$ is of exponential type. To finish, if in addition $\supp(a)$ is locally finite, then so is $\spec(A/B)$. Since $\a$ has exponential growth, we obtain that $\Theta$ has locally finite spectrum, thus so $\p'$ has locally finite spectrum. The previous applications of Lemmas \ref{lem:aprpoisson} and \ref{lem:expprim} (and their content) show that $E$ has locally finite spectrum as well, and so Lemma \ref{lem:order1apc} shows that $E$ is a trigonometric polynomial.
\end{proof}

We believe that Theorem \ref{thm:5} should still hold only assuming that $\supp(\mu)$ is locally finite. Note that by Remark \ref{rem:exptodeg2}, $\supp(\mu)$ is contained in a finite union of uniformly discrete sets. Also, the uniformly discreteness of $\supp(\mu)$ plays a role only in Step 3 above, where we show $\p'$ is bounded as a stepping stone to show the zeros of $E$ cannot get close to the real axis. Note that if $\p'$ is bounded then $\supp(\mu)$ is uniformly discrete. However, we still think there must be a way to circumvent this issue, and nevertheless conclude that the zeros of $E$ are separated from the real axis. Unfortunately we were not able to come up with such maneuver, despite many efforts.

\begin{remark}
We now construct one example of an $\fs$-pair $(\mu,a)$ satisfying the conclusion of Theorem \ref{thm:5} for which $E$ is not a trigonometric polynomial. We let $E=B'-iB$ where
$$
B(z) = \sum_{n=0}^\infty 10^{-n} \sin(\pi (1-2\theta_n) z),
$$
where $\theta_0=0$ and $\{\theta_n\}_{n\geq 1}\subset (0,1/2)$ is a sequence of irrational numbers such that $\{1,\theta_1,...,\theta_k\}$ is independent over $\z$ for each $k\geq 0$. Assume also that $\theta_1=\inf_{n\geq 1} \theta_n$. Note that $B$ is well-defined entire function and almost periodic in $\cp$ since its series converges uniformly in any horizontal strip. Also note that $B$ is of exponential type at most $\pi$. In order to show that $E$ is Hermite-Biehler we only need to verify that $B$ has only real zeros. Consider the box $Q_M=[-M-1/2,M+1/2] + i[-M-1/2,M+1/2]$, for $M\geq 0$. Then function $\sin(\pi z)$ has exactly $2M+1$ simple real zeros in this box. Since, 
$$
|\sin(z)|^2 = \cosh(y)^2 \sin(x)^2 + \sinh(y)^2 \cos(x)^2 \leq e^{2|y|},
$$
it is not hard to show that 
$|\sin(\pi z)| >\bigg| \sum_{n=2}^\infty 10^{-n} \sin(\pi \si_n z)\bigg|$
on the boundary of $Q_M$. Rouche's Theorem implies that $B(z)$ has the same number of zeros in $Q_M$ as $\sin(\pi z)$ for large $M$, which equals $2M+1$. However, $|B(n+1/2)-(-1)^n|\leq \frac19$ for all $n\in \z$, and so $B$ has only real simple zeros, with exactly one zero in each interval $(n-1/2,n+1/2)$ for all $n$. Also note that $$\spec(iB'/B)=\{k_1\theta_1 + ...+k_l\theta_l +m_1(1-\theta_1)+...+m_l(1-\theta_l) : k_j,m_j\geq 0\},$$ and so $\spec(iB'/B) \cap (0,\theta_1)=\emptyset$.
Moreover, by construction, each $\la\in \spec(iB'/B)$ has a uniquely associated pair of vectors $(k_1,....,k_l)$ and $(m_1,...,m_l)$ such that $\la=k_1\theta_1 + ...+k_l\theta_l +m_1(1-\theta_1)+...+m_l(1-\theta_l)$. We obtain that
$$
\E(iB'/B)(\la) = 10^{-(k_1+m_1+...+k_l+m_l)} \geq 10^{-\la/\theta_1}.
$$
Since $iB'/B\in \apc$ we obtain that 
$$
\sum_{0\leq \la<T} |\E(iB'/B)(\la)| \leq 10^{T/\theta_1}e^{4\pi T}\sum_{\la\in \r} |\E(iB'/B)(\la)|^2e^{-4\pi \la}=10^{T/\theta_1}e^{4\pi T}\E[|B'/B(\cdot+i)|^2]<\infty.
$$
Hence $\la\mapsto \E(iB'/B)(\la)$ is locally summable. We conclude that if we define the $\fs$-pair $(\mu,a)$ as in  \eqref{def:amu} (with $A=B'$) then
$$
\mu=2\pi \sum_{B(\ga)=0} \del_\ga
$$
has uniformly discrete support and its Fourier transform is supported in $\pm \spec(iB'/B)$. Hence it satisfies the conclusion of Theorem \ref{thm:5}, but $E$ is not a trigonometric polynomial. Finally, let $K_0=\spec(iB'/B)$ and $K_{p}=K_{p-1}'$ for $p\geq 1$, where $K_{p-1}'$ is the set of accumulation points of $K_p$. Assume now that $\lim \theta_n=1/2$. It is then simple to deduce, by the description of $K_0$, that $K_p = K_0+\frac12\{p,p+1,p+2,...\}$, and so the support of $\ft \mu$ is not dense in any interval.
\end{remark}

\subsection{de Branges spaces  - a Hilbert space interpretation of $\fs$-pairs}
A de Branges space \cite{dB} (see also the introduction of \cite{GL}) is a Hilbert space $(\h,\|\cdot\|)$ of entire functions $F:\cp\to\cp$ satisfying:
\begin{enumerate}
\item[(H1)] If $F\in \h$ and $F(w)=0$ for some $w\in \cp$, then $G(z)=F(z)\frac{z-\w}{z-w}$ belongs to $\h$ and $\|G\|=\|F\|$;
\item[(H2)] The functional $F\in \h\mapsto F(w)$ is continuous for every $w\in \cp$;
\item[(H3)] If $F\in \h$ then $F^*\in \h$ and $\|F^*\|=\|F\|$.
\end{enumerate}
Because of (H2), the space $\h$ comes equipped with an unique reproducing kernel $K(w,z)$, that is, $F(w)=\langle F, K(w,\cdot) \rangle $ for any $w\in \cp$ and $F\in \h$.
De Branges proved \cite[Thm. 23]{dB} that for any such Hilbert space $\h$ there exists an Hermite-Biehler function $E$ such that $\h=\h(E)$ isometrically, where $\h(E)$ is the Hilbert space of entire functions $F$ such that\footnote{The original definition involves bounded type theory, but this is an equivalent short definition, see \cite{GL}.}
$$
\|F\|^2:=\sup_{y\in \r} \int_\r \left|\frac{F(x+iy)}{E(x+i|y|)}\right|^2 \d x < \infty,
$$
in which case the sup above is attained at $y=0$. De Branges also shows \cite[Thm. 19]{dB} that any such space $\h(E)$ is a Hilbert space that satisfies the  above axioms.  When we identify $\h=\h(E)$  for some Hermite-Biehler function $E=A-iB$ we have
\begin{align}\label{repker}
K(w,z)=\frac{B(z)\ov{A(w)}-\ov{B(w)}{A(z)}}{\pi(z-\w)} = \frac{E(z)\ov{E(w)}-E^*(z)E(\w)}{2\pi i(\w-z)}
\end{align}
and
$$
F(w)=\int_\r \frac{F(t)\ov{K(w,t)}}{|E(t)|^2}\d t
$$
for every $w\in \cp$ and $F\in \h$. The function $E$ realizing $\h=\h(E)$ is not unique. Indeed, if $\h(E)=\h(E_1)$ isometrically, one can then apply \cite[Problem 69]{dB} (for $S=A_1$ and $S=B_1$) to obtain
$$
E_1(z) = \frac{e^{i\be}}{\sqrt{1-|p|^2}}(E(z)-\ov{p}E^*(z))
$$
for some $\be\in \r$ and $p\in \cp$, with $|p|<1$. Conversely, a routine computation shows that any $E_1$ defined in the above way satisfies \eqref{repker}, and since $0<K(z,z)=(4\pi y)^{-1}(|E_1(z)|^2-|E_1^*(z)|^2)$ for $z\in \H$, we conclude that $E_1$ is Hermite-Biehler, and so $\h(E)=\h(E_1)$ isometrically. Moreover, their theta functions are related by the M\"obius transformation
$$
\Theta_1(z) = e^{2i\be}\frac{\Theta(z)-p}{1-\Theta(z)\ov{p}} .
$$ 

A major result \cite[Thm. 22]{dB} is that, for any $\al\in [0,\pi)$, the set $\{K(\ga,z)\}_{\p(\ga)\equiv \al \, (\text{mod } \pi)}$ is orthogonal in $\h(E)$, and its orthogonal complement is one-dimensional and spanned by $e^{i\al}E-e^{-i\al}E^*$. In particular, if $\al=0$ and $E$ has no real zeros, and $B\notin \h(E)$, then representation of functions in the basis $\{K(\ga,z)\}_{B(\ga)=0}$ implies the identities
$$
\pi K(w,z):=\frac{B(z)\ov{A(w)}-\ov{B(w)}{A(z)}}{z-\w} = \sum_{B(\ga)=0} \frac{1}{\p'(\ga)} \frac{B(z)\ov{B(w)}}{(\ga-z)(\ga-\w)}
$$
and
$$
F(z) = \sum_{B(\ga)=0} \frac{F(\ga)B(z)}{B'(\ga)(z-\ga)}, \quad \text{for any } F\in \h(E),
$$
with convergence in $\h(E)$ and uniformly (and absolutely) in compact sets of $\cp$. Letting $f=iA/B$,  if one divides the expansion of $K$ above by $iB(z)\ov{B(w)}$ we get
$$
\frac{f(z)+\ov{f(w)}}{z-\w} = \frac{1}{2\pi i} \int_\r \frac{\d \mu(t)}{(t-z)(t-\w)}
$$
with $\mu=\sum_{B(\ga)=0} \frac{2\pi}{\p'(\ga)}\del_\ga$.
Assume now that $f=iA/B \in \apc$ and that $\E f(\cdot)$ is locally summable.  Thus, Theorem \ref{thm:HBmupos} applies and $(\mu,a)$ is a $\fs$-pair (with $\a$ as in \eqref{def:amu}). There is no particular reason (other than convenience) to take $\al=0$ here. One could select as well any $\al\in [0,\pi)$, but now the $\fs$-pair would be\footnote{The proof of Theorem \ref{thm:maineven} implies that $(e^{i\al}E-e^{-i\al}E^*)\notin \h(E)$}
\begin{align}
\begin{split}\label{genmuapair}
\mu &= \sum_{\p(\ga)\equiv \al \, (\text{mod } \pi)} \frac{2\pi}{\p'(\ga)}\del_\ga, \\
a(\la) &= \ov{a(-\la)} = \E f (\la) \ (\la >0) \\
 a(0)&=2\re \E f(0),
 \end{split}
\end{align}
assuming $f=(1+e^{-2i\al}\Theta)/(1-e^{-2i\al}\Theta) \in \apc$ (note if $E_\al=e^{i\al}E=A_\al-iB_\al$ then $f=iA_\al/B_\al$).

One could ask the question whether there is another natural axiom, as the ones above, that would force the space $\h$ to ``produce" an $\fs$-pair. Such axiom indeed exists.

\begin{proposition}
Let $(\h,\|\cdot\|)$ be a Hilbert space of entire functions satisfying (H1), (H2) and (H3), with reproducing kernel $K(w,z)$. Assume also that:
\begin{enumerate}
\item[(H4)] There is $\al \in \H$ such that the function
$$
f(z)=\frac{{(\ov{\al}-z)K(\al,z)}-(\al-z)K(\ov{\al},z)}{{(\ov{\al}-z)K(\al,z)}+(\al-z)K(\ov{\al},z)}
$$
belongs to $\apc$ and $\E f(\cdot)$ is locally summable.
\end{enumerate}
Then there is  an Hermite-Biehler function $E=A-iB$ such that $\h=\h(E)$ isometrically, and, if we define $\mu$ and $\a$ as in \eqref{def:amu}, we have that $(\mu,a)$ is a $\fs$-pair.
\end{proposition}
\begin{proof}
We can then apply \cite[Thm. 23, p. 58]{dB} to obtain that the identity
\begin{align*}
\frac{(K(w,z)-K(\be,z)K(\be,\be)^{-1}K(w,\be))(z-\ov{\be})}{z-\be} = \frac{(K(w,z)-K(\ov{\be},z)K(\ov{\be},\ov{\be})^{-1}K(w,\ov{\be}))(\w-\ov{\be})}{\w-\be}
\end{align*}
holds for all $w,z\in \cp$ and $\be\in \H$. Let $L(w,z)=2\pi i (\w-z)K(w,z)$ and
$$
E_1(z)=L(\al,\al)^{-1/2}L(\al,z).
$$
A routine computation, using the above identity, now shows that 
$$
K(w,z) = \frac{E_1(z)\ov{E_1(w)}-E_1^*(z)E_1(\w)}{2\pi i(\w-z)}.
$$
Since $0<K(z,z)$ for $z\in \H$, we conclude that $E_1$
is of Hermite-Biehler class, thus $\h=\h(E_1)$ isometrically. Letting $E_1=A_1-iB_1$ and noting that $L(\al,z)^*=-L(\ov{\al},z)$,  we deduce that
$$
f(z)=\frac{1+\frac{L(\al,z)^*}{L(\al,z)}}{1-\frac{L(\al,z)^*}{L(\al,z)}}=\frac{1+\Theta_1(z)}{1-\Theta_1(z)}=i\frac{A_1(z)}{B_1(z)}.
$$
We can then apply Theorem \ref{thm:HBmupos} to finish the proof.
\end{proof}

\section{Almost Periodic Functions in $\r$}\label{sec:AP}
We say that a continuous function $f:\r\to\cp$ in almost periodic (in the sense of Bohr \cite{Bo}) if for every $\ep>0$ the following set is relatively dense
$$
\tau_\ep(f):=\{t\in \r : \sup_{x\in \r} |f(x)-f(x+t)|\leq \ep\}.
$$
This set is called the set of $\ep$-translations for $f$. 
We denote by $\apr$ the set of continuous almost period functions $f:\r\to\cp$. It is not hard to show that any $f\in \apr$ is bounded and uniformly continuous, and that $\tau_\ep(f)$ is also closed with non-empty interior. The following is a very useful criteria for almost periodicity (see \cite[p. 7]{AP}). In what follows $C(\r)$ denotes the usual Banach algebra of bounded continuous functions with the topology induced by the sup-norm $\|f\|_{\infty}=\sup_{x\in \r} |f(x)|$. 

In the remaining part of this section we compile necessary facts about almost periodic functions and we provide some proofs for completeness.

\begin{theorem}[Bochner's criterion]
Let $f:\r\to\cp$ be continuous. Then $f\in \apr$ if and only if the set of functions $\{f(\cdot +h)\}_{h\in \r}$ is pre-compact in $C(\r)$.
\end{theorem}

\begin{remark}
Bochner's criterion works also for functions $f:\r\to \cp^n$. In particular, if $f_j\in \apr$ for $j=1,...,N$, then $\cap_{j=1}^N \tau_\ep(f_j)$ is nonempty and relatively dense.
\end{remark}

It is now straightforward to show that $\apr$ closed under multiplication, addition and uniform convergence, hence $\apr$ is a closed subalgebra of $C(\r)$. Moreover, it also follows that for every continuous $g:\cp\to\cp$ we have $g\circ f \in \apr$ whenever $f\in \apr$. Since exponentials $e^{2\pi i x \la}$, with $\la \in \r$, are periodic, we conclude by Bochner's criterion that any \emph{trigonometric polynomial }
$$
p(x)=\sum_{n=1}^N a_n e^{2\pi i \la_n x} \quad \quad (\la_n\in \r)
$$
belongs to $\apr$.

\begin{lemma}\label{lem:polyaprox}
Given any $f\in \apr$ and $\ep>0$, there is a trigonometric polynomial $p$ such that $\|f-p\|_\infty<\ep$. In particular, $\apr$ is the closure in $C(\r)$ of the algebra of trigonometric polynomials. 
\end{lemma}
\begin{proof}
We can assume $\|f\|_\infty\leq 1$. For a given $\ep>0$ let $\tau_\ep$ be the set of $\ep$-translations for $f$.  First we will need some auxiliary functions. We claim that for every sufficiently large $M\geq 1$ there exists $\p_M\in C^\infty_c(\tau_\ep \cap (-M,M))$ such that $0\leq \p_M\leq 1$, $\int_\r \p_M =1$, $\int_\r |\p_M'| =O_{\ep}(1)$ and 
$$
\sum_{n\in \z} |\ft \p(n/(4M))|^{1/2} = O_{\ep}(1).
$$
Assuming this claim is true, we now finish the proof. Let
$$
f_M(x) = \int_\r f(x+t)\p_M(t)\d t
$$
and note that $\|f-f_M\|_\infty\leq \ep$ for any $M$. Also note that if $|x|\leq M$ then
$$
f_M(x)=  \int_{-2M}^{2M} f(t)\p_M(t-x)\d t. 
$$
The function $\p_M(x)$ can be identified with a $4M$-periodic  $C^\infty$-function and so we have
$$
\p_m(x)=\sum_{n\in\z} \theta_n e^{2\pi i n x/(4M)}
$$
where $\sum_{n\in \z} |\theta_n|^2 = (4M)^{-1}\int_{-2M}^{2M} |\p_M(x)|^2\d x$ and 
$$
\theta_n = (4M)^{-1}\int_{-2M}^{2M} \p_M(x)e^{-2\pi i n x/(4M)}\d x = \ft \p_M(n/(4M))/(4M).
$$
We obtain
$$
f_M(x) = \sum_{n\in\z} \wt \theta_n e^{-2\pi i n x/(4M)}  \quad \text{ for } |x|\leq M,
$$
where $\wt \theta_n = \theta_n \int_{-2M}^{2M}f(t) e^{2\pi i n t/(4M)}\d t$. Note this representation converges absolutely since $ |\wt \theta_n| \leq  4M|\theta_n| = |\ft \p_M(n/(4M))| \leq 1$, $\sum_{n\in\z} |\wt \theta_n|^{1/2} = O_{\ep}(1)$, and so, $\sum_{n\in\z} |\wt \theta_n|=O_{\ep}(1)$. Now enumerate $(\wt \theta_n)_{n\in \z}$ in decreasing order of magnitude and call this new sequence $(\al_{M,n})_{n\geq 1}$.  We obtain
$$
f_M(x) = \sum_{n\geq 1} \al_{M,n} e^{2\pi i \la_{M,n} x}  \quad \text{ for } |x|\leq M,
$$
for some $\la_{M,n}\in \frac{1}{4M}\z$. Noticing that 
$$
|\al_{M,n}|^{1/2}n \leq \sum_{j=1}^n |\al_{M,j}|^{1/2} = O_{\ep}(1),
$$
we obtain $|\al_{M,n}|\leq n^{-2}$. A standard Cantor's diagonal procedure guarantees the existence of $(\al_n)_{n\geq 1}$ such that $\sum_{n\in\z} |\al_n| = O_{\ep}(1)$ and a subsequence $M_k\to\infty$ such that $\lim_k \al_{M_k,n}=\al_n$ for all $n\geq 1$.  Let now $I\subset \n$ be the $n$'s such that $\sup_{k} |\la_{M_k,n}| < \infty$.  By further taking a subsequence of the $M_k$'s we can assume that $\lim_k \la_{M_k,n} \to \la_n$ for $n\in I$.  Note now that
$$
|\al_{M_k,n}| \leq |\ft \p_M(-\la_{M_k,n})|  \leq  \frac{1}{2\pi |\la_{M_k,n}|} \int_{\r} |\p_M'(x)|\d x  = O_\delta(|\la_{M_k,n}|^{-1}).
$$
In particular, $\lim_k \al_{M_k,n} = \al_n=0$ if $n\notin I$. Finally, the uniform bound $|\al_{M,n}|\leq n^{-2}$ forces $f_{M_k}$ to converge uniformly in compact sets to 
$$
g(x) =  \sum_{n\geq 1} \al_{n} e^{2\pi i \la_{n} x},
$$
which in particular implies that
$$
\|f-g\|_\infty \leq \ep.
$$
However, we can now simply truncated $g$ to find a trigonometric polynomial $p$ such that $\|f-p\|_\infty \leq 2\ep$.

It remains to construct the auxiliary functions $\p_M$. Since $f$ is also uniformly continuous, it is easy to show that $\{t_n\}_{n\in \z} + (-\delta,\delta) \subset \tau_{\ep/4}$ for some small $\delta=\delta_\ep>0$ and some sequence $\{t_n\}_{n\in \z}$ satisfying $1/\delta \leq t_{n+1}-t_n \leq 3/\delta$ for all $n$. Take $h \in C^\infty_c(-1,1)$ even with $0\leq h\leq 1$ and $\int_\r h=1$. Let $h_\delta(x)=h(x/\delta)/\delta$, $\psi(x)=(2N+1)^{-1}\sum_{n=-N}^N h_\delta(x-t_n)$ and $\p_M(x)=\psi * \psi *\psi * \psi$, where $N\geq 1$ is selected to be the largest such that $\{t_n\}_{|n|\leq N} + (-\delta,\delta) \subset \tau_{\ep/4}\cap (-M/4,M/4)$. Note we must have $N\geq \frac{\delta M}{10}$ for sufficiently large $M$. Since $\tau_{\ep/4} + \tau_{\ep/4}+\tau_{\ep/4}+\tau_{\ep/4}\subset \tau_\ep$ we conclude that $\p_M \in C^\infty_c(\tau_\ep \cap (-M,M))$. Also note that $\int_\r \p_M = (\int_\r \psi)^4 = 1$ and, since $\p_M'=\psi'*\psi*\psi *\psi$, that 
$$
\int_\r |\p'_M| \leq \int _\r|\psi'|  = O(1/\delta).
$$
Finally note that 
\begin{align*}
\sum_{n\in\z} |\ft \p_M(n/(4M))|^{1/2} = \sum_{n\in\z} |\ft \psi(n/(4M))|^{2} & = 4M \int_{-2M}^{2M}|\psi(x)|^2\d x \\ & = \frac{4M}{\delta (2N+1)} \int_{-1}^1 |h(x)|^2 \d x  \leq \frac{4M}{\delta (2N+1)}  \ll \delta^{-2},
\end{align*}
for $M$ large enough. This finishes the proof.
\end{proof}

We define
$$
\E f := \lim_{T\to\infty} \frac{1}{2T}\int_{-T}^T f(x)\d x \quad \text{and} \quad \E f(\la):=\lim_{T\to\infty} \frac{1}{2T}\int_{-T}^T f(x)e^{-2\pi i \la x}\d x,
$$
whenever these limits exist (so $\E f = \E f(0)$). 

\begin{lemma}\label{lem:propAP}
For any $f\in \apr$ we have:
\begin{enumerate}
\item The average $\E f(\la)$ exists;
\item The set
$
\spec(f) := \{\la\in \r : \E f(\la) \neq 0\}
$
it at most countable;
\item If $\spec(f)=\emptyset$ then $f=0$;
\item If $\sum_{\la\in \r} |\E f(\la)|<\infty$, then the following series converges absolutely and uniformly
$$
f(x)=\sum_{\la\in \r} \E f(\la)e^{2\pi i \la x};
$$
\item We have $\E |f|^2 = \sum_{\la} |\E f(\la)|^2$ and for every $\ep>0$ there is $S\subset \spec(f)$ finite such that $\E_x|f(x)-\sum_{\la\in S}\E f(\la)e^{2\pi i \la x }|^2<\ep$.
\end{enumerate}
\end{lemma}
\begin{proof}
For item (1), we can assume $\la=0$. Let $\ep>0$ be given and take a trigonometric polynomial $p$ such that $\|f-p\|_\infty\leq \ep$. Let $A=\E p$, which exists by direct computation. We have
\begin{align*}
\frac{1}{2T}\int_{-T}^T f(x)\d x - \frac{1}{2T'}\int_{-T'}^{T'} f(x)\d x & = O(2\ep) + \frac{1}{2T}\int_{-T}^T p(x)\d x - \frac{1}{2T'}\int_{-T'}^{T'} p(x)\d x \\ & = O(2\ep)  + A - A + o_{T,T'}(1).
\end{align*}
For item (2), note first that $\E f(\la)$ exists since $f(x)e^{-2\pi i \la x}$ is also almost periodic (by Bochner's criterion). Let $p_n$ be a trigonometric polynomial such that
$
\|f-p_n\|_\infty < \frac{1}{n}.
$
Assume that $|\E f(\la)|>0$ and let $1/n<|\E f(\la)|$. We obtain that
$|\E p_n(\la)| \geq |\E f(\la)|-\frac{1}{n} > 0$, hence $\E p_n(\la)\neq 0$. 
We conclude that $\spec(f) \subset \bigcup_{n\geq 1} \spec(p_n)$.
For item (3), let $\ep>0$ be given and take a trigonometric polynomial $p$ such that $\|\bar f-p\|_\infty\leq \ep$. Now, observing that
$$
\frac{1}{2T}\int_{-T}^T |f(x)|^2\d x = \frac{1}{2T}\int_{-T}^T f(x)p(x)\d x + \frac{1}{2T}\int_{-T}^T f(x)(\bar f(x)-p(x))\d x = o_T(1) + O(\ep\|f\|_\infty),
$$
we conclude that $\lim_T \frac{1}{2T}\int_{-T}^T |f(x)|^2\d x=0$. This is impossible if $f$ is nonzero, because $|f|^2$ is also almost periodic and uniformly continuous, and so if $|f(x)|^2 \geq c$ for $x$ in some interval $I$ then $|f(x)|^2 \geq c/2$ for $x\in \tau_{c/2}(|f|^2)\
 + I$. Since $\tau_{c/2}(|f|^2)$ contains some increasing sequence $\{t_n\}_{n\in \z}$ satisfying $t_{n+1}-t_n =O(1)$, it implies that $\E |f|^2 \geq \E ({\bf 1}_{\tau_{c/2}(|f|^2) + I})>0$, which is absurd. We conclude that $f=0$.
For item (4), note that 
$$
g(x)=\sum_{\la\in \r} \E f(\la)e^{2\pi i \la x},
$$
is well-defined and converges absolutely and uniformly on $\r$. Also note that dominated convergence implies that $\E g(\la)=\E f(\la)$ for all $\la\in \r$, and so by item (3) $f=g$.
For item (5), let $ \Lambda':=\cup_{m\geq 1} \spec(p_m)$ for some trigonometric polynomials $p_m$ such that $\|f-p_m\|_\infty<1/m$ and enumerate $\Lambda'=\{\la_n\}_{n\geq 1}$. It is enough to show that $\E |f|^2 = \sum_{n\geq 1} |\E f(\la_n)|^2$. Let $f_N(x) = \sum_{n=1}^N \E f(\la_n) e^{2\pi i \la_n x}$. A straightforward computation shows that $\E |f|^2 - \E |f_N|^2 = \E |f-f_N|^2$ and so $\sum_{n=1}^N |\E f(\la_n)|^2 \leq \E |f|^2$ for all $N$. Another routine computation shows
$$
\E |f-p|^2 = \E |f|^2 - \E |f_N|^2 + \E |f_N-p|^2 \geq \E |f|^2 - \E |f_N|^2,
$$
for any trigonometric polynomial of the form $p(x)=\sum_{n=1}^N b_n e^{2\pi i \la_n x}$, for some $b_n\in \cp$. Since $p_m$ is of such form for some $N_m$, we conclude that $\E |f-f_{N_m}|^2 < 1/m$ and that 
$$
\E |f_{N_m}|^2 \leq \E |f|^2 \leq \E |f_{N_m}|^2 +2/m.
$$
\end{proof}
The following proposition  (see \cite[p. 46]{Be}) justifies the notation 
$$
f(x)\sim \sum_{\la\in\r} \E f(\la) e^{2\pi i \la x}.
$$

\begin{proposition}[Bochner's approximation]\label{prop:bochapprox}
For any $f\in \apr$ there is an sequence of functions $c_n: \r \to [0,1]$, $c_n\leq c_{n+1}$, each $c_n$ with finite support, satisfying 
$$
\lim_{n\to\infty} c_{n}(\la) = \begin{cases} 1 &\mbox{if }\ \E f(\la)\neq 0 \\ 
 0 &\mbox{if }\  \E f(\la)= 0,
\end{cases}
$$
and such that 
$$
\lim_{n\to\infty} \sup_{x\in \r} |f(x) - \sum_{\la\in\r} c_n(\la)\E f(\la) e^{2\pi i \la x}|=0.
$$
\end{proposition}

\section{Almost Periodic Functions in $\H$}\label{sec:apc} Recall that we have defined in the introduction the space $\apc$ of holomorphic almost periodic functions $f:\H\to\cp$ such that for every $\ep>0$ there is a relatively dense set of $\ep$-translations $\tau_\ep(f) \subset \r$ satisfying
$$
\sup_{\ep < \Im z < 1/\ep} |f(z)-f(z+t)|<\ep
$$ 
for every $t\in \tau_\ep(f)$. We say a function $f:\H \to \cp$ is \emph{bounded on strips} if $$\sup_{\ep < \Im z < 1/\ep} |f(z)|<\infty$$ for every $\ep>0$. We now give an alternative characterization of almost periodicity.

\begin{lemma}\label{lem:altdefapc}
Let $f : \H \to \cp$ be holomorphic. Then following are equivalent:
\begin{enumerate}
\item $f$ is bounded on strips and for every $h>0$ we have $f(\cdot +ih)\in \apr$;
\item $f$ is bounded on strips and there is $h>0$ such that $f(\cdot + ih) \in \apr$;
\item $f\in \apc$;
\end{enumerate}
In this case the quantity
$$
\E f(\la) := \lim_{T\to\infty} \frac{1}{2T}\int_{-T+iy}^{T+iy} f(z)e^{-2\pi i \la z}\d z
$$
exists for every $\la\in\r$, is independent of $y>0$, it is nonzero in at most countably many $\la$'s and
$$
\sum_{\la\in\r}|\E f(\la)|^2e^{-4\pi y \la} = \E [|f(\cdot +iy)|^2]
$$
for every $y>0$.
\end{lemma}

\begin{proof}
The implications $(1)\Rightarrow (2)$ and $(3)\Rightarrow (1)$ are obvious. We now show $(2)\Rightarrow (3)$. Let $\tau_\ep$ be the set of $\ep$-translations for $f(\cdot + ih)$. We claim that $\tau_\ep$ also works in any horizontal strip containing the line $\Im z=h$. Indeed, let $t\in \tau_\ep$ and $0<y_1=y_2/2<y_2<h$. For all $y$ with $y_2<y<h$, we can then apply Hadamard's there lines lemma (for $\Im z\in \{y_1,y,h\}$) to conclude that
\begin{align*}
& \sup_x \log |f(x + iy)-f(x + iy+t)| \\ & \leq \frac{y-y_1}{h-y_1}\sup_x \log |f(x + ih)-f(x + ih+t)| +\frac{h-y}{h-y_1} \sup_x \log |f(x + iy_1)-f(x + iy_1+t)| \\
& \leq \frac{y_2-y_1}{h-y_1}\log \ep +\frac{h-y_2}{h-y_1} B_1
\end{align*}
for  some $B_1>0$. Hence, there is $\al=\al(y_2,h)>0$ such that
$$
\sup_{y_2<\Im z<h} |f(z)-f(z+t)| \leq \frac{1}{\al}\ep^\al.
$$
We can apply the same procedure for $h<\Im z<y_3$ for any $y_3>h$ to prove the claim. This shows that $f\in \apc$. Now notice that by Lemma \ref{lem:propAP} the limit $\E f(\la)$ exists for every $y>0$. To show is independent of $y$, we can use Cauchy's formula to deduce that if $0<y_1<y_2$ then
\begin{align*}
&\bigg[\frac{1}{2T}\int_{-T+iy_1}^{T+iy_1} f(z)e^{-2\pi i \la z}\d z -  \frac{1}{2T}\int_{-T+iy_2}^{T+iy_2} f(z)e^{-2\pi i \la z}\d z \bigg]  \\ & = 
\bigg[\frac{1}{2T}\int_{-T+iy_1}^{-T+iy_2} f(z)e^{-2\pi i \la z}\d z -  \frac{1}{2T}\int_{T+iy_1}^{T+iy_2} f(z)e^{-2\pi i \la z}\d z \bigg] = O((y_2-y_1)e^{2\pi |\la|y_2}/T).
\end{align*}
Taking $T\to\infty$, we conclude that $\E f(\la)$ is independent of $y$. Since $e^{-2\pi \la y}\E f(\la) = \E [f(\cdot +i y)](\la)$ we deduce that $\la \in \r \mapsto \E f(\la)$ has countable support and  $\sum_{\la\in \r}|e^{-2\pi \la y}\E f(\la)|^2 =  \E |f(\cdot +i y)|^2$. This finishes the proof.
\end{proof}

For a function $f\in \apc$ (or $f\in \apr$) we define the \emph{spectrum} of $f$ to be 
$$
\spec(f):=\{\la\in \r : \E f(\la)\neq 0\}.
$$

\begin{lemma}\label{lem:fdisc}
The following hold:
\begin{enumerate}
\item[(i)] Let $f\in \apc$. Then $f$ is bounded in $\H +ic$, for some $c>0$, if and only if $\spec(f)\subset [0,\infty)$. In this case $f$ is bounded in $\H +ic$ for any $c>0$ and $\lim_{y\to \infty} \sup_x |f(x+iy)-\E f(0)|=0$;
\item[(ii)] If $f\in\apc$, $\spec(f)$ is bounded from below and $f(z)\neq 0$ for all $\im z>c$, for some $c>0$, then $\inf \spec(f)\in \spec(f)$;
\item[(iii)] If $f,g\in\apc$ have spectrum bounded from below  then so has $fg\in \apc$. If in addition $f,g$ have locally finite spectrum, then so has $fg$.
\item[(iv)] If  $f\in \apc$ and $f\neq 0$ in $\H$, then $\inf_{\ep < \im z <1/\ep} |f(z)|>0$ for all $\ep>0$ and $1/f\in \apc$. Moreover, if $f$ has spectrum bounded from below  then so has $1/f$, and if in addition $f$ has locally finite spectrum then so has $1/f$. 
\item[(v)] If $f\in \apc$ and $|f(z)|<1$ for all $z\in \H$ then for every $\ep>0$ there is $c>0$ such that $|f(z)|<1-c$ if $\im z>\ep$.
\end{enumerate}
\end{lemma}
\begin{proof}
\noindent \underline{Items (i) and (ii)}: These are direct applications of \cite[Thm. p. 162 \& p. 152]{Be}. \underline{Item (iii)}: Note that $f_1(z)=e^{-2\pi i M z}f(z)$ and $g_1(z)=e^{-2\pi i M z}g(z)$ have spectrum contained in $[0,\infty)$ for some $M>0$, thus both are bounded for $\Im z>1$ and so is $h=f_1g_1$. Thus, $\spec(fg)=\spec(h)+2M\subset [0,\infty)$, hence $fg$ have spectrum bounded from below. We also have $\spec(fg)\subset\spec(f)+\spec(g)=\spec(f_1)+\spec(g_1)-2M$, and since sums of locally finite sets contained in $[0,\infty)$ is also locally finite and contained in $[0,\infty)$, we conclude that the spectrum of $fg$ is locally finite if both $f$ and $g$ have locally finite spectrum. \underline{Item (iv)}: By a clever application of Hadamard's Three-Lines theorem  \cite[ Thm. 11, p. 139 \& Thm. 9, p. 144]{Be}, one can show that if $w$ is an accumulation point of $f$ in an horizontal line, then $f(z)=w$ has a solution in any horizontal strip containing this line.  In particular, if $f$ never vanishes in $\H$ then $|f|$ is bounded away from zero in any horizontal strip, thus $1/f\in \apc$. If $f\sim \sum_{\la\geq M} \E f(\la)e^{2\pi i\la z}$ where $M=\inf \spec(f)$, 
since $f$ has no zeros, $M\in \spec(f)$, $p=\E f(M)\neq 0$ and $\lim_{y\to \infty} \sup_x |f(x+iy)e^{-2\pi iM (x+it)}-p|=0$. In particular $g(z)=1-f(z)e^{-2\pi iM z}/p$ is bounded in absolute value by $1/2$ for $\im z> c$, for some large $c>0$ and $\spec(g)\subset (0,\infty)$. We obtain
$$
\frac{1}{f(z)} = \frac{e^{-2\pi iM z}}{p}\frac{1}{1-g(z)} = \frac{e^{-2\pi iM z}}{p} \sum_{n \geq 0} g(z)^n,
$$
where the sum converges absolutely and uniformly for $\Im z>c$. We conclude that $\E (1/f)(\la) = \sum_{n \geq 0} \E(g^n)(\la+M)$ (with absolute convergence), and so $\E(1/f)(\la)=0$ if $\la<-M$. If in addition $\spec(f)$ is locally finite, then $\spec(g)$ is locally finite and contained in $[\delta,\infty)$, for some $\delta>0$. Since $\spec(1/f)+\{M\}\subset\cup_{n\geq 0}\big(\oplus^n\spec(g)\big)$, we conclude that $(\spec(1/f)+\{M\})\cap [0,T] \subset \cup_{0\leq n \leq T/\delta}\big(\oplus^n\spec(g) \cap [0,T]\big)$, and so $\spec(1/f)$ is locally finite. \underline{Item (v)}: Since $f$ is bounded we deduce that $\E f(\la)=0$ for $\la<0$ and, by \cite[Thm. 9, p. 144]{Be}, that $\sup_{\ep < \im z <1/\ep} |f(z)|< 1$ for every $\ep>0$. Hence $|\E f(0)|<1$. On the other hand $\lim_{y\to \infty} \sup_x |f(x+iy)-\E f(0)|=0$.
\end{proof}

\begin{lemma}\label{lem:expprim}
Let $f\in \apc$ and assume $\spec(f)\subset \{0\}\cup [b,\infty)$ for some $b>0$. If $F:\H\to\cp$ is holomorphic and $iF'/F=f$, then $F\in \apc$ and  $\spec(F)\subset \{-\E f(0)/(2\pi)\} \cup  [-\E f(0)/(2\pi)+b,\infty)$. Moreover, if in addition $\spec(f)$ is locally finite, then so it is $\spec(F)$.
\end{lemma}
\begin{proof}
Let $p=\E f(0)$. Since $F$ has no zeros in $\H$ we can write $F(z)=e^{-ig(z)-ipz}$ for some $g:\H\to\cp$ holomorphic. Since $g'(z)+p=iF'(z)/F(z)=f(z)$ we conclude that $g'\in \apc$, $\spec(g)\subset [b,\infty)$ and $\E (g')(\la)=\1_{\la\geq b} \E f(\la)$.  We can now apply \cite[Thm. 9, p. 152]{Be} to conclude that $g\in \apc$ and that $\E g(\la)=\1_{\la\geq b} \E f(\la)/(2\pi i \la)$ for $\la \neq 0$. If $\Phi:\cp\to\cp$ is entire and $h\in \apc$ then it is easy to conclude that $\Phi\circ h\in  \apc$. Indeed, since $h$ is bounded in strips and since $\Phi'$ is bounded in bounded sets, we conclude that for any horizontal strip contained in $\H$ there is $C>0$ such that $|\Phi(h(z))-\Phi(h(w))| \leq C|h(z)-h(w)|$ for all $z$ and $w$ in that strip. This shows that $\Phi\circ h\in \apc$. This implies that the power series expansion
$$
\Phi(h(z)) = \sum_{n\geq 0} \frac{\Phi^{(n)}(0)}{n!}h(z)^n.
$$
converges absolutely and uniformly on any horizontal strip. Now assume that $\spec(h)\subset [b,\infty)$ for some $b>0$. Since $\spec(h^n)\subset [nb,\infty)$ and
$$
\E [\Phi \circ h](\la) = \sum_{n\geq 0} \frac{\Phi^{(n)}(0)}{n!}\E[h^n](\la),
$$
we conclude that $\spec(\Phi \circ h) \subset \{0\}\cup [b,\infty) $. If in addition $\spec(h)$ is locally finite, since $$\spec(\Phi \circ h) \cap [0,T] \subset \cup_{0\leq n \leq T/b}\big(\oplus^n\spec(h)\cap [0,T]\big),$$ we conclude that $\spec(\Phi \circ h)$ is locally finite.
Considering $\Phi(z)=e^{-iz}$ and $h(z)=g(z)-\E g(0)$, we deduce that $F(z)=\Phi(h(z))e^{-i\E g(0)-ipz}$ belongs to $\apc$ and that $\spec(F)\subset \{-p/(2\pi)\} \cup  [-p/(2\pi)+b,\infty)$. Moreover, $\spec(F)$ is locally finite if $\spec(f)$ also is.
\end{proof}

We also have approximation by trigonometric polynomials (see \cite[p. 148]{Be})
\begin{proposition}[Bochner's approximation]\label{prop:bochapprox2}
Let $f\in \apc$. Then for every $b>1$ there is a sequence of functions $c_n:\r\to[0,1]$ as in Proposition \ref{prop:bochapprox} such that
$$
\lim_{n\to\infty} \sup_{1/b < \Im z < b} |f(z) - \sum_{\la\in\r} c_n(\la)\E f(\la) e^{2\pi i \la z}|=0.
$$
\end{proposition}

\section{Representations of Analytic Functions}\label{sec:rep}
We say a holomorphic function $g:\H\to\cp$ if of \emph{bounded type} if there are two bounded holomorphic functions $P,Q:\H \to\cp$ such that $g=P/Q$. We write $g\in \bt(\H)$. If $g$ is of bounded type then the \emph{mean type}
$$
\vartheta(g):=\limsup_{y\to\infty} \frac{\log|g(iy)|}{y}
$$
exists and is finite. For more information about these definitions we recommend \cite{dB,Ga}.

\begin{lemma}\label{lem:condatinfty}
Let $f:\H\to\cp$ be holomorphic and write $g=\exp(f)$. The following are equivalent:
\begin{enumerate}
\item There is a real-valued locally finite measure $\mu$ with $\deg(\mu)\leq 2$ such that
\begin{align}\label{eq:fkernelrep1}
\frac{f(z)+\ov{f(w)}}{z-\ov w} =  \frac{1}{2\pi i}\int_\r  \frac{\d \mu(t)}{(t-z)(t-\w)}
\end{align}
holds for every $z,w\in \H$;
\item $g\in \bt(\H)$ and $\vartheta(g)=0$.
\end{enumerate}
\end{lemma}
\begin{proof}
First we show that $(1)\Rightarrow (2)$. Indeed, note first that
$$
\Re f(x+iy) = \frac{y}{\pi }\int_\r  \frac{\d \mu(t)}{(t-x)^2+y^2}.
$$
Writing $\mu=\mu_1-\mu_2$ where each $\mu_j$ is nonnegative and defining 
$$
f_j(z) = \frac{1}{\pi i} \int_\r \frac{1+tz}{t-z}\frac{\d \mu_j(t)}{1+t^2}
$$
we conclude that $\Re f_j \geq 0$ in $\H$ and that $\Re f = \Re f_1 - \Re f_2$, hence $f=f_1-f_2+ih$ for some constant $h\in \r$. However, since each $g_j=\exp(-f_j)$ is bounded in absolute value by $1$ in $\H$ and 
$g=e^{ih} g_2/g_1$, we deduce that $g\in \bt(\H)$. Note also that
$$
\vartheta(g) = \limsup_{y\to\infty} \frac{\Re f(iy)}{y} = \limsup_{y\to\infty} \frac{1}{\pi }\int_\r  \frac{\d \mu(t)}{t^2+y^2} = 0.
$$
Now we show $(2)\Rightarrow (3)$. Since $g$ has no zeros, Nevalinnas's factorization for functions of bounded type \cite[Theprem 9]{dB} implies that there exists a unique real-valued locally finite measure $\mu$, with $\deg(\mu)\leq 2$, and some $c,h\in \r$ such that
$$
e^{f(z)}=g(z) = e^{-ihz} \exp\bigg( ic +  \frac{1}{\pi i} \int_\r \frac{1+tz}{t-z}\frac{\d \mu(t)}{1+t^2} \bigg).
$$
Since $h=\vartheta(g)=0$, a simple computation shows that \eqref{eq:fkernelrep1} holds.
\end{proof}

\begin{lemma}\label{lem:aprpoisson}
Let $f\in \apr$ and assume that $\spec(f)\cap (0,b)=\emptyset$ for some $b>0$. Then
$$
F(z)= \frac{1}{2\pi i} \int_\r \frac{1+tz}{t-z}\frac{f(t)\d t}{1+t^2} 
$$
belongs to $\apc$ and $\E F(\la)=\1_{\la\geq b}\E f(\la)$ for $\la \neq 0$.
\end{lemma}
\begin{proof}
Since any $f\in \apr$ is bounded, the integral defining $F$ converges absolutely and defines an holomorphic function $F$. Noticing that $\frac{1+tz}{(t-z)(1+t^2)} = \frac{1}{t-z} - \frac{t}{1+t^2}$ we obtain that
$$
F'(z) = \frac{1}{2\pi i} \int_\r \frac{f(t)\d t}{(t-z)^2},
$$
and so $|F'(x+iy)|\leq \frac{1}{2y}\max_{t\in \r} |f(t)|$ for $y>0$, hence $F'$ is bounded on strips. Since
$$
F'(x+i) = \frac{1}{2\pi i} \int_\r \frac{f(t+x)\d t}{(t-i)^2}
$$
it is clear that any set of $\ep$-translations for $f$ is also one for $F'(\cdot +i)$, and so $F'(\cdot +i) \in\apr$, and Lemma \ref{lem:altdefapc} shows that $F'\in \apc$. Dominated Convergence implies that
$$
\E F'(\la) = e^{2\pi \la} \E f(\la) \frac{1}{2\pi i} \int_\r  \frac{e^{2\pi i \la t}}{(t-i)^2} = e^{2\pi \la} \E f(\la) 2\pi i |\la| e^{-2\pi |\la|}\1_{\la>0} = 2 \pi i \max\{\la,0\}\E f(\la) ,
$$
and, since $\spec(f)\cap (0,b)=\emptyset$ for some $b>0$, we obtain $\spec(F')\subset [b,\infty)$. We can now apply \cite[Thm. 9, p. 152]{Be} to conclude that $F\in \apc$.
\end{proof}

\begin{lemma}\label{lem:order1apc}
Let $f: \cp\to\cp$ be entire such that $f, f^*\in \apc$ and both have spectrum bounded from below. If $f$ has finite order, that is,
$$
{\rm order}(f):= \inf \{\rho>0 : \limsup_{|z|\to\infty} \frac{\log^+ |f(z)|}{|z|^{\rho}} < \infty\} < \infty,
$$
then $f$ has finite exponential type, that is, 
$$
\tau=\limsup_{|z|\to\infty} \frac{\log |f(z)|}{|z|} < \infty.
$$
In this case $f$ and $f^*$ are of bounded type, $f$ is bounded on the real line and $|f(z)| \leq M e^{\tau|y|}$ for all $z=x+iy$, where $M=\sup_{x\in \r} |f(x)|$ and $\tau=\max\{\vartheta(f^*),\vartheta(f)\}$. Moreover,  $\E f(\la)=\ov{\E (f^*)(-\la)}$ for all $\la\in \r$ and 
$$
\spec(f)=-\spec(f^*)\subset [-\tau/(2\pi),\tau/(2\pi)].
$$
If in addition $\spec(f)$ is locally finite, then $f$ is trigonometric polynomial.
\end{lemma}

\begin{proof}
Observe that since $f$  has spectrum bounded from below, then Lemma \ref{lem:fdisc}(i) shows that for $m=\min\{\inf \spec(f),\inf \spec(f^*)\}$ we have that $|f(z)e^{-2\pi i m z}|+|f^*(z)e^{-2\pi i m z}|$ is bounded for $\im z>c$, for any $c>0$. This in particular shows that $f(z)e^{-2\pi |m|  z}$ is bounded on the sides of any sector $S_c=\{z=x+iy : x >|y|/c\}$. Since $f$ has finite order, one can apply Phranm\"en-Lindel\"of for sectors (with small $c>0$) \cite[Thm. 1, p. 37]{L2} to conclude that $f(z)e^{-2\pi |m| z}$ is bounded in $S_c$. We then apply the same argument replacing $f(z)$ by $f^*(-z)$ to conclude that $f^*(-z)e^{-2\pi |m| z}$ is bounded in $S_c$. All these imply that $f$ has finite exponential type, that is,
$$
|f(z)| \ll e^{C|z|}
$$
for some $C>0$. However, since $f(z+i)e^{-2\pi i m (z+i)}$ is bounded on the real line, another application of Phranm\"en-Lindel\"of \cite[Thm. 3, p. 38]{L2} shows that $|f(z+i)e^{-2\pi i m (z+i)}| \ll e^{C'|y|} $ for some $C'>0$, and so $|f(z)|\ll e^{C''|y|}$ for some $C''>0$. In particular $f$ is bounded on every horizontal strip and, by the same result, we can take $C'=\tau$. A straightforward calculation using Cauchy's integral formula shows that
$$
\E f(\la)-\ov{\E (f^*)(-\la)} = \lim_{T\to\infty}\frac{1}{2T} \left( \int_{T-i}^{T+i} f(z)e^{-2\pi i \la z}\d z - \int_{-T-i}^{-T+i} f(z)e^{-2\pi i \la  z}\d z \right)=0,
$$
since $f$ is now bounded on horizontal strips. In particular, $\spec(f)=-\spec(f^*)$. A well-known result of Krein \cite{Kr} (see also \cite[Thm. 1, p. 115]{L2}) shows that $f,f^*\in \bt(\H)$ and that $\tau=\max\{\vartheta(f^*),\vartheta(f)\}$. Since $g(z)=f(z)e^{i \tau z}$ is bounded in $\H$, Lemma \ref{lem:fdisc}(i) shows that 
$\spec(f) +\tau/(2\pi)= \spec(g)\subset [0,\infty)$,
hence $\spec(f)\subset [-\tau/(2\pi),\infty)$. The same argument shows that $\spec(f^*)\subset [-\tau/(2\pi),\infty)$, but since $\E f(\la)=\ov{\E (f^*)(-\la)}$, we obtain that $\spec(f)\subset [-\tau/(2\pi),\tau/(2\pi)]$. Finally, if $\spec(f)$ is locally finite, Proposition \ref{prop:bochapprox2} proves straightforwardly that $f$ is a trigonometric polynomial.
\end{proof}

\section{Proof of Theorem \ref{thm:maineven}}

We start we some lemmas.

\begin{lemma}
 For $w,z\in \H$ let
\begin{align}\label{gspecial}
g(w,z,x)=\frac{e^{-2\pi i \w |x|}\1_{x<0} + e^{2\pi i z |x|}\1_{x\geq 0}}{z-\w}.
\end{align}
Then 
$$
\ft g(w,z,\xi) =\frac{1}{2\pi i(\xi-z)(\xi-\w)}.
$$
\end{lemma}
\begin{proof}
Letting $f_z(x) = e^{2\pi i z |x|}\1_{x> 0} + \frac{1}{2}\1_{x=0}$ we have
$$
\ft f_z(\xi) = \int_0^\infty e^{2\pi i x(z-\xi)}\d x = \frac{1}{2\pi i (\xi-z)}.
$$
The lemma follows since $g(w,z,x) = (z-\w)^{-1}(f_{-\w}(-x)+f_z(x))$.
\end{proof}

\begin{lemma}\label{lem:limexist}
Let $(\mu,a)$ be a $\fs$-pair with $\deg(\mu)\leq 2$. Then 
\begin{align}\label{crucialid}
\lim_{T\to\infty} \sum_{|\la|<T}a(\la)g(w,z,\la)(1-|\la|/T) = \frac{1}{2\pi i}\int_\r  \frac{\d \mu(t)}{(t-z)(t-\w)}
\end{align}
uniformly for $w$ and $z$ in the region
$$
R_c:=\{z\in \H : |\Re z| \leq  1/c \text{ and } \Im z \geq c\}, \quad (\text{for any } c>0).
$$
\end{lemma}
\begin{proof}
Let $\p\in C_c^\infty(\r)$, $\int \p = 1$, $\p\geq 0$, $\supp \p \subset (-1,1)$ and define $\p_\ep(x)=\p(x/\ep)/\ep$ for $0<\ep<1$. Let $w,z\in R_c$, $T>2/c$ and
$$
g_{\ep,T}(x)  : = (g(w,z,\cdot)(1-\tfrac{1}T|\cdot|)_+) * \p_\ep(x),
$$
where $s_+=\max\{0,s\}$. Note that $g_{\ep,T}\in C_c^\infty(-\ep-T,T+\ep)$ and
$$
\ft g_{\ep,T}(\xi) = (\ft g(w,z,\cdot) * {S_T}(\xi)) \ft \p(\ep \xi),
$$
where $S_T(x)=\frac{\sin^2(T \pi x)}{T (\pi x)^2}$. Note that $|\ft \p(\ep \xi) |\leq 1$. We claim that\footnote{We write $A\ll_p B$ if there is a numerical constant $K>0$, depending only in the parameter $p$, such that $A\leq K B$.}
\begin{align}\label{eq:unifbound}
|\ft g(w,z,\cdot) * {S_T}(\xi) | \ll_c \frac{1}{\xi^2+c^2}.
\end{align}
We shall prove this claim in the end. Since $(\mu,a)$ is a $\fs$-pair, in particular we obtain
$$
\sum_{|\la|<{T+\ep}}a(\la)g_{\ep,T}(\la) = \int_\r \ft g_{\ep,T}(t)\d\mu(t).
$$
Since $\a$ is locally summable, and as $\ep\to0$, $g_{\ep,T}(x)\to g(w,z,x)(1-|x|/T)_+$ uniformly in $x\in \r$ and $\ft g_{\ep,T}(\xi) \to \ft g(w,z,\cdot) * S_T(\xi)$ pointwise, Dominated Convergence implies that
$$
\sum_{|\la|<T}a(\la)g(w,z,\la)(1-|\la|/T) =  \int_\r \ft g(w,z,\cdot) * {S_T}(t)\d\mu(t).
$$
It is now enough to show the right hand side above converges uniformly in the region $(w,z)\in R_c^2$.  This can be done by noticing that 
$$
|\ft g(w,z,\xi_1)-\ft g(w,z,\xi_2)| = \frac{1}{2\pi |z-\w|}\bigg|\frac{1}{\xi_1-z}-\frac{1}{\xi_1-\w}-\frac{1}{\xi_2-z}+\frac{1}{\xi_2-\w}\bigg| \leq \frac{|\xi_1-\xi_2|}{4\pi c^3},
$$
and so $\ft g(w,z,\cdot) * {S_T}(\xi) \to \ft g(w,z,\xi)$ uniformly in $x\in \r$ and $(w,z)\in R_c$, as $T\to\infty$. Finally, the uniform bound \eqref{eq:unifbound} guarantees the desired uniform convergence. It remains to prove the claim \eqref{eq:unifbound}. First note that 
$$
|\ft g(w,z,\xi)|  \ll \frac{1}{(\xi-\Re z)^2+c^2} + \frac{1}{(\xi-\Re w)^2+c^2} \ll_c \frac{1}{c^2+\xi^2},
$$
and so we can use a contour change  (and that $T>2/c$) to obtain 
\begin{align*}
|g(w,z,\cdot) * {S_T}(\xi)|  & \ll_c \frac1{T} \int_{\r} \frac{\sin^2(\pi T t)}{((\xi-t)^2+c^2)t^2}\d t  = \frac1{T} \int_{\r} \frac{\sin^2(\pi T (t+i/T))}{((\xi-t-i/T)^2+c^2)(t+i/T)^2}\d t \\
& \ll \frac1{T} \int_{\r} \frac{1}{((\xi-t)^2+c^2)(t^2+1/T^2)}\d t = \frac{\pi(c+1/T)/c}{\xi^2 + (c+1/T)^2} \ll \frac{1}{c^2+\xi^2}.
\end{align*}
\end{proof}

\vspace{-1cm}

\begin{proof}[\bf Proof of Theorem \ref{thm:maineven} (and Theorem \ref{cor:1})]
$\nonumber$

\noindent {\it \underline{Necessity}.}  
Assume that $(\mu,a)$ is an real-antipoal $\fs$-pair such that $\deg(\mu)\leq 2$ and $\a$ has exponential growth. Since $a(-\la)=\ov{a(\la)}$, we can apply Lemma \ref{lem:limexist} (multiplying both sides by $(z-\w)$) to obtain
\begin{align*}
& a(0) + \lim_{T\to\infty} \bigg(\sum_{0<\la<T} a(\la)e^{2\pi i  \la z}(1-\la/T)  + \sum_{0<\la<T} \ov{a(\la)}e^{-2\pi i \la \w}(1-\la/T) \bigg) \\ & = \frac{1}{2\pi i}\int_\r  \bigg(\frac{1}{t-z}-\frac{1}{t-\w}\bigg){\d \mu(t)} = \frac{1}{2\pi i}\int_\r  \frac{1+tz}{t-z}\frac{\d\mu(t)}{1+t^2} - \frac{1}{2\pi i}\int_\r  \frac{1+t\w}{t-\w}\frac{\d\mu(t)}{1+t^2},
\end{align*}
where the limit exists and is uniform for $w$ and $z$ in compact sets of $\H$. In particular, we conclude there exists $h\in\r$ such that
\begin{align}\label{eq:fPoissonrep}
f(z) = \frac12 a(0) + \lim_{T\to\infty} \sum_{0<\la<T} a(\la)e^{2\pi i  \la z}(1-\la/T) = ih+ \frac{1}{2\pi i}\int_\r  \frac{1+tz}{t-z}\frac{\d\mu(t)}{1+t^2},
\end{align}
where the limit above exists and is uniform in compact sets. A simple computation shows
\begin{align}\label{eq:fkernelrep}
\frac{f(z)+\ov{f(w)}}{z-\ov w} =  \frac{1}{2\pi i}\int_\r  \frac{\d \mu(t)}{(t-z)(t-\w)}.
\end{align}
This proves assertion (I).  Since $\a$ has exponential growth, we let $c=\inf \{b>0 : \sum_{\la\in \r} |a(\la)|e^{-2\pi b |\la|} < \infty\}$. Then we must have
$$
f(z) = \frac12 a(0) +\sum_{\la>0} a(\la)e^{2\pi i  \la z}
$$
for $\Im z>c$, as the above series converges absolutely and uniformly for $\Im z > c + \ep$, for any $\ep>0$. Thus, after choosing an enumeration for $\supp(a)$, we conclude that $f(\cdot +ih)$ is the uniform limit of trigonometric polynomials and we obtain that $f(\cdot +ih) \in \apr$ for every $h>c$.  By Lemma \ref{lem:altdefapc} we obtain $f(\cdot+ic)\in \apc$. This proves assertion (II). Assertion (III) is direct since if $p(x)=\sum_{\theta \in S} p_\theta e^{2\pi i \theta x}$, for some finite set $S\subset [0,\infty)$, then
\begin{align*}
\limsup_{T\to\infty}\bigg| \frac{1}{2T}\int_{-T}^T f(x+2ic)\ov{p(x)}\d x\bigg| & = \bigg| \frac12a(0)\del_S(0) + \sum_{\la>0} a(\la)e^{-4\pi c\la} \del_S(\la) p_\la \bigg| \\
& \ll \max |p_\theta| \sum_{\la\in \r} |a(\la)|e^{-4\pi c |\la|} \ll \max |p_\theta|.
\end{align*}
Finally, by Lemma \ref{lem:condatinfty}, assertion (IV) is implied by representation \eqref{eq:fkernelrep}.

\noindent {\it \underline{Sufficiency}.}  We prove sufficiency replacing property $({\rm III})$ by $({\rm III}^*)$ as in Theorem \ref{cor:1}, and so also proving Theorem \ref{cor:1}. Let $c_1=\inf\{b>0 : f(\cdot + ib)\in \apc\}$ and $f_1(z)=f(z+c_1)$. By condition $({\rm II})$, we can apply Lemma \ref{lem:altdefapc} to deduce the limit
$$
\E f_1(\la)=\lim_{T\to\infty} \frac{1}{2T}\int_{-T+iy}^{T+iy} f_1(z)e^{-2\pi i \la z}\d z =e^{-2\pi \la c_1} \lim_{T\to\infty} \frac{1}{2T}\int_{-T+iy+ic_1}^{T+iy+ic_1} f(z)e^{-2\pi i \la z}\d z 
$$
exists, does not depend on $y>0$ and is nonzero only for countable many real $\la$'s. We also have that $\sum_{\la} |\E f_1(\la)|^2e^{-4\pi h \la} < \infty$ for all $h>0$. This shows that 
$$
\E f(\la) = \lim_{T\to\infty} \frac{1}{2T}\int_{-T+iy}^{T+iy} f(z)e^{-2\pi i \la z}\d z
$$
is independent of $y>c_1$ and is nonzero only for countable many real $\la$'s. By condition $({\rm III}^*)$, there is $B_M>0$ and $c_2\geq c_1$ (with equality only if $f(\cdot+ic_1)\in \apr$) such that
$$
B_M \max |p_\theta| \geq \limsup_{T\to\infty}\bigg| \frac{1}{2T}\int_{-T}^T f(x+ic_2)\ov{p(x)}\d x\bigg| = \bigg| \sum_{0\leq \la\leq M} \E f(\la)e^{-2\pi c_2\la} \del_S(\la) p_\la \bigg|
$$
whenever $p(x)=\sum_{\theta \in S} p_\theta e^{2\pi i \theta x}$ for some finite set $S\subset [0,M]$. We conclude that 
$\spec(f)\subset [0,\infty)$ and
$$
\sum_{0\leq \la\leq M} |\E f(\la)| \leq B_M e^{2\pi c_2 M}
$$
for every $M>0$, so the function $a(\la)$ defined in \eqref{eq:adef} is locally summable. Lemma \ref{lem:fdisc}(i) $f$ is bounded for $\im z>2c_1$, in particular $\lim_{y\to\infty} \re f(iy)/y=0$. By condition $({\rm IV})$, we can now apply Lemma \ref{lem:condatinfty} to deduce there is a real-valued locally finite measure $\mu$ of degree at most $2$ such that \eqref{eq:fkernelrep} is true. In particular,
$$
2\Re f(x+iy) = \frac{y}{\pi }\int_\r  \frac{\d \mu(t)}{(t-x)^2+y^2}.
$$
As in Lemma \ref{lem:condatinfty} one can then write $f=f_1-f_2$ such that $\Re f_j \geq 0$ in $\H$ and $f_j(i)=\int_{\r} \frac{\d \mu_j(t)}{2\pi(1+t^2)}$. We can then apply Harnack's inequality \cite[p. 136]{Be}, that translates to
$$
\bigg|\frac{f_j(z)-f_j(i)}{f_j(z)+f_j(i)}\bigg| \leq \bigg|\frac{z-i}{z+i}\bigg|.
$$
Since 
$$
\frac{|z+i|+|z-i|}{|z+i|-|z-i|}\leq \frac{(1+|z|)^2}{4y} \quad (z=x+iy),
$$
we conclude that $|f_j(z)| \leq  \frac{(1+|z|)^2}{4y}f_j(i)$ and so $|f(z)|\leq  \frac{(1+|z|)^2}{4y}\int_{\r} \frac{\d |\mu|(t)}{2\pi(1+t^2)}$. We deduce that
$$
\liminf_{r\to\infty}\frac{1}{r} \int_0^\pi \log^+|f(re^{i\theta})|\sin \theta \d \theta = 0,
$$
and so one can now apply the Phranm\"en-Lindel\"of principle (as in \cite[Thm. 1]{dB}) to obtain that $f(z+ic)$ is bounded in $\H$ for every $c>c_1$. Lemma \ref{lem:fdisc}(i) shows that $\E f(\la)=0$ for $\la<0$. Finally, let us prove that $(\mu,a)$ is a $\fs$-pair. For $z\in \H$ and $t\in \r$ let $P_z(t)=\frac{z}{\pi i(t^2-z^2)}$. Since identity \eqref{eq:fkernelrep} holds true, a simple computation shows
$$
{f(z+s)+\ov{f(-\bar z +s)}} = P_z * \mu(s)
$$
for all $s\in\r$. By Bochner's approximation on $\apr$, for every $h>0$, one can find a sequence of functions $d_n:\r\to[0,1)$, each with finite support and such that $\lim_{n} d_n(\la)={\bf 1}_{\spec(f)}(\la)$, and that
$$
\lim_n \sup_{\Im z = c_1+h} |f(z) - \sum_{\la \geq 0} \E f(\la)d_n(\la) e^{2\pi i \la z}| = 0.
$$
Let $\p \in C_c^\infty((-M,M))$ be antipodal. It is now straightforward (using Bochner's approximation, Dominated Convergence and local summability) to conclude that
$$
\int_{\r} f(z+s)\ft \p(s) \d s = \sum_{0\leq \la < M} \E f(\la) \p(\la) e^{2\pi i \la z}.
$$
if $\Im z> c_1$. Another routine computation (using that $a(0)=2\re \E f(0)$) shows
\begin{align*}
\sum_{|\la|<M} a(\la) \p(\la) e^{2\pi i |\la| z} = \int_\r [{f(z+s)+\ov{f(-\bar z +s)}}] \ft \p(s)\d s &= \int_{\r} P_z * \mu(s)\ft \p(s) \d s \\
& = \int_{\r} P_z * \ft \p(t) \d \mu(t)
\end{align*}
if $\Im z> c_1$. By real-antipodal splitting and linearity, the above formula holds for all $\p\in C_c^\infty(\r)$. Now observe that both sides extend analytically to $\Im z >0$. Taking $z=iy$ and noting that $P_{iy}$ is Poisson's kernel, a routine application of Dominated Convergence and approximate identity arguments (using that $\a$ is locally summable and $\deg(\mu)\leq 2$) allows us to take $y\to 0$ and conclude that
$$
\sum_{|\la|<M} a(\la) \p(\la)  =\int_\r \ft \p(t) \d \mu(t).
$$
Therefore $(\mu,a)$ is an $\fs$-pair. This finishes the proof.
\end{proof}

\section{$\fs$-pairs, an account of various examples.}\label{sec:exemp}
Throughout this section we will define certain families of $\fs$-pairs for which we will try to provide as many examples as possible, old and new. We define these families so that they are closed under scaling, translation, modulation and multiplication by scalar.

\subsection{Finite Support}
Let $(\mu,a)$ be a $\fs$-pair such that $\a$ has finite support, that is, there is $\{p_n\}_{n=1}^N\subset \cp$ and $\{\la_n\}_{n=1}^N\subset \r$ such that
$$
\int_\r \ft \p(t) \d\mu (t) = \sum_{n=1}^N p_n \p(\la_n)
$$
for all $\p\in C^\infty_c(\r)$. It is easy to see that we must have $\d\mu (t) = \sum_{n=1}^N p_n e^{2\pi i \la_n t}\d t$. 

\subsection{Uniformly Discrete Poisson Summation ($\udps$)} We denote $\udps$ the set of $\fs$-pairs $(\mu,a)$ such that there is $\al>0$, $\{\theta_j\}_{j=1}^N\subset [0,\alpha)$ and trig-polynomials $\{Q_j\}_{j=1}^N$ where
$$
\mu= \sum_{j=1}^N \sum_{\la \in \al \z+\theta_j} Q_j(\la) \del_\la.
$$
It is easy to see that if $(\mu,a)\in \udps$ then $\a$ is also of the above form. The breakthrough of Lev \& Olveskii \cite{LO2} shows that $(\mu,a)\in \udps$ if and only if $\mu$ and $\a$ have uniformly discrete support (assuming apriori that both are strongly tempered), which justifies the name of this family.  We observe  that we can always assume that all $Q_j's$ are equal since by classical interpolation results (for instance, Lagrange interpolation via Chebyshev polynomials) one can always find trigonometric polynomials $P_j$, which are $\alpha$-periodic, and such that $P_j(\theta_k)=\del_{j,k}$. Thus, if we let $Q = \sum_{j=1}^N P_j Q_j$ then $Q(\al n + \theta_j)=Q_j(\al n +\theta_j)$ for all $n\in \z$, and so
$$
\mu = \sum_{\la \in \cup_{j=1}^N (\al \z+\theta_j)} Q(\la) \del_\la.
$$

\subsection{Real Rooted Trigonometric Polynomial ($\rrtp$)}
We let the class $\rrtp$ to be the family of $\fs$-pairs $(\mu,a)$ such that there are four trigonometric polynomials $E_j$ ($j=1,2,3,4$) of Hermite-Biehler class, that is, $|E_j(z)|>|E_j(\ov z)|$ for $\im z>0$, and such that
\begin{align*}
\mu &= \sum_{\, \p_1(\ga)\equiv 0 \,({\rm mod} \, \pi)} \frac{Q_1(\ga)}{\p_1'(\ga)} {\del}_{\ga}-\sum_{\, \p_2(\ga)\equiv 0 \,({\rm mod} \, \pi)} \frac{Q_2(\ga)}{\p_2'(\ga)} {\del}_{\ga} \\
& -  i\sum_{\, \p_3(\ga)\equiv 0 \,({\rm mod} \, \pi)} \frac{Q_3(\ga)}{\p_3'(\ga)} {\del}_{\ga}+i\sum_{\, \p_4(\ga)\equiv 0 \,({\rm mod} \, \pi)} \frac{Q_4(\ga)}{\p_4'(\ga)} {\del}_{\ga},
\end{align*}
where $Q_j$ are trigonometric polynomials, $\p_j$ is the phase function associated with $E_j$, defined by the condition that $E_j(x)e^{i\p_j(x)}$ is real for all $x\in \r$ ($\p_j$ is uniquely defined modulo $\pi$, see \eqref{def:hermitebiehler}). We observe that if we write $E_j=A_j-iB_j$, were $A_j,B_j$ are trigonometric polynomials which are real on the real line, then both $A$ and $B$ have only real zeros, $A/B$ have only simple real zeros and simple poles which interlace, 
 $$
\{\ga\in \r : \, \p_j(\ga)\equiv 0 \,({\rm mod} \, \pi)\} = {\rm Zeros}(B_j/A_j) \quad \text{and} \quad  \frac1{\p_j'(\ga)} = \res{\ga}(A_j/B_j)>0.
 $$
By Remark \ref{rem:KS}, $\rrtp$ contains the pairs $(\mu,\ft \mu)$ where
$$
\mu=\sum_{P(\ga)=0} m(\ga)Q(\ga)\del_{\ga}
$$
and $Q$ and $P$ are any given trigonometric polynomials, $P$ with only real roots and $m(\ga)$ is the multiplicity of the zero $\ga$. These measures were first introduced and studied by Kurasov and Sarnak \cite{KS}. Observe that $ \udps \subset \rrtp$ since $\cup_{j} \al \z + \theta_j=\text{Zeros}(\prod_j \sin(\pi(z-\theta_j)/\al))$. Indeed $\rrtp$ is a much richer class than $\udps$ since, for instance, $$P(z)=\det(U + \text{diag}(e^{2\pi i l_1 z},...,e^{2\pi i l_n z}))$$ has only real roots if $U$ is a $n\times n$ unitary matrix and $l_1,....,l_n \in \r$. The zeros of such polynomials are generically not contained in a finite union of arithmetic progressions \cite{AV}. Trigonometric polynomials with only real zeros have been classified recently in \cite{ACV} as restrictions of Lee-Yang polynomials. A polynomial $R:\cp^n\to\cp$ is Lee-Yang if $R(z_1,...,z_n)\neq 0$ whenever $\min\{\max\{|z_1|,...,|z_n|\},\max\{|z_1|^{-1},...,|z_n|^{-1}\}\}<1$.  Indeed, the main result in \cite{ACV} states that if $P$ is a trigonometric polynomial with only real zeros then there exists a Lee-Yang polynomial $R(z_1,...,z_n)$ and positive reals $l_1,...,l_n$, linear independent over $\q$, such that 
$$
P(x)=R(e^{il_1x},...,e^{il_n x}).
$$
A simple example is
\[R(z_1,z_2)=\det\left(\text{Id}_3- \begin{bmatrix}
   0& -1/3 & 4/3 \\1 & 0& 0 \\ 0&  2/3 &1/3
\end{bmatrix}\text{diag}(z_1,z_1,z_2)\right) = 1+z_1^2/3-z_2(1/3+z_1^2)
\]
and
$$
P(z) = \frac{3i}2R(e^{\pi i (\sqrt{2}-1) z},e^{2 \pi i(1+ \sqrt{2}) z})=\sin(2\pi z)+3\sin(2\pi \sqrt{2} z).
$$
Note the characteristic polynomial of the matrix above is $(x-1)(3x^2+2x+3)/3$, with roots $\{1,-1/3\pm i\sqrt{8}/3\}$, and so it is an unitary matrix. Hence, the polynomial $R$ above is Lee-Yang and $P$ has only real roots.

\subsection{Crystalline Measures ($\cm$)}
We let $\cm$ be the class of $\fs$-pairs $(\mu,a)$ such that both $\mu$ and $\a$ have locally finite support. By Theorem \ref{thm:HBmupos} we conclude that $\rrtp\subset \cm$. One main theme investigated in \cite{KS,AV,ACV} is to find pairs $(\mu,a)\in \cm$ such that $\supp(a)$ and/or $\supp(\mu)$ intersects any infinite arithmetic progression at most finitely\footnote{So very non-periodic}. Guinand \cite[p. 265]{Gui}, in 1959, was ``almost the first" to produce such pair, as we shall explain below. He constructs a (self-dual) $\fs$-pair $(\mu,\mu)$ with
$$
\mu = \sum_{n\geq 0} c_n (\del_{\sqrt{n+1/9}}+\del_{-\sqrt{n+1/9}}),
$$
where $c_n$ are the cofficients of the modular form\footnote{Guinand did not realized this nor constructed his example in the way we do here.}
\begin{align}\label{eq:guiexa}
\begin{split}
\frac{\eta(z)^{2/3}\eta(4z)^{2/3}}{\eta(2z)^{1/3}} &= q^{1/9}\left(1
 - \frac{2}{3}\*q
 - \frac{4}{9}\*q^2
 - \frac{40}{81}\*q^3
 - \frac{160}{243}\*q^4
 + \frac{268}{729}\*q^5
 + \frac{1808}{6561}\*q^6 + \cdots\right) \\
& = \sum_{n\geq 0} c_n q^{1/9+n}.
\end{split}
\end{align}
Above
$$
\eta(z)=q^{1/24}\prod_{n\geq 1}(1-q^n) = \sum_{n\geq 1} \chi_{12}(n)q^{n^2/24}
$$
is the Dedekind eta-function, $q=e^{2\pi i z}$ and $\chi_{12}$ is the Dirichlet character of modulus $12$ with $\chi_{12}(1)=\chi_{12}(11)=-\chi_{12}(5)=-\chi_{12}(7)=1$ and zero otherwise.
Note that if $n=9m^2+2m$ then $\sqrt{n+1/9}=3m+1/3$, and so the support of $\mu$ contains an infinite arithmetic progression $\{3m+1/3\}_{m\geq 1}$.
We will show below that, with the right point of view, one can embed Guinand's example and also classical Poisson summation in a real one-parameter family, for which almost all members have the property that their support intersect any infinite arithmetic progression at most four times. For that we need first the following lemma.

\begin{lemma}\label{lem:FG}
Let $F\in \apc$, with locally finite spectrum contained in $[0,\infty)$ and assume that 
$$
\sum_{n\geq 0} |\E F(\ga_n)|(1+\ga_n)^{-k}<\infty
$$
for some $k>0$, where $\spec(F)=\{\ga_n\}_{n\geq 0}$. Assume that $G(z)=F(-1/z)\sqrt{i/z}$ also belongs to $\apc$,  has locally finite spectrum contained in $[0,\infty)$ and $\E G$ satisfies a similar estimate the one above for $\E F$. Let
$$
\mu = \sum_{n\geq 0} \E F(\ga_n)(\del_{\sqrt{2\ga_n}}+\del_{-\sqrt{2\ga_n}}).
$$
Then $\mu$ is a strongly tempered measure and
$$
\ft \mu = \sum_{n\geq 0} \E G(\la_n)(\del_{\sqrt{2\la_n}}+\del_{-\sqrt{2\la_n}})
$$
where $\spec(G)=\{\la_n\}_{n\geq 0}$. Hence $(\mu,\ft \mu)\in \cm$.
\end{lemma}

\begin{proof}
The proof is inspired by \cite[Lemma 2.1]{CG}. Consider the gaussian $g_z(t)=e^{\pi i z t^2}$ for $z\in\H$. Observe that $\ft g_z(t)=\sqrt{i/z}g_{-1/z}(t)$ and that $g_z(\sqrt{2s})=e^{2\pi i z s}$ for $s\geq 0$.  Then the identity
\begin{align*}
G(z)=\sum_{n\geq 0} \E G(\la_n)  g_z(\sqrt{2\la_n}) = F(-1/z)\sqrt{i/z} = \sum_{n\geq 0} \E F(\ga_n) \ft g_{z}(\sqrt{2\ga_n})
\end{align*}
shows that
$$
\sum_{n\geq 0} \E F(\ga_n) (\ft \p(\sqrt{2\ga_n})+\ft \p(-\sqrt{2\ga_n})) = \sum_{n\geq 0} \E G(\la_n) (\p(\sqrt{2\la_n})+\p(-\sqrt{2\la_n})),
$$
for any $\p\in {\rm span}_{\cp}\{g_z : z\in \H\}$. An approximation argument similar to the one employed in \cite[Lemma 2.1]{CG} shows that the above identity actually holds for any $\p\in \mathcal{S}(\r)$.
\end{proof}

The following constructions were inspired by a similar one partially communicated in 2018 by Danylo Radchenko in a seminar at the University of Bonn.

\subsubsection{\underline{$N$-Level eta-quotients and self-dual crystalline measures}}  Let $N\geq 1$ be an integer and $\br = \{r_d\}_{d|N}$ be a sequence of reals indexed by the divisors of $N$ such that
\begin{align}\label{eq:rcond}
r_d=r_{N/d}, \quad \sum_{d|N}r_d = 1 \ \text{ and } \  \frac{1}{24}\sum_{d|N}dr_d = b,
\end{align} 
where $c\geq 0$. Consider the eta-quotient
$$
\eta(\br,z) = \prod_{d|N} \eta(dz)^{r_d} = q^{b}\sum_{n\geq 0}\al_n q^{n}.
$$
The eta-function $\eta:\H\to\cp$ is an holomorphic function that satisfies the functional identities $\eta(z+\ell)=e^{\pi i  \ell /12}\eta(z)$, for $\ell\in\z$, and $\eta(-1/z)=\sqrt{z/i}\, \eta(z)$. In particular we obtain that 
\begin{align*}
\eta(\br,z+1) = e^{2\pi i b}\eta(\br,z)
\end{align*}
\begin{align*}
\eta(\br,-1/z) = \prod_{d|N} \eta(-d/z)^{r_d} = \prod_{d|N}(z/(di))^{r_d/2} \prod_{d|N} \eta(z/d)^{r_d} & = \sqrt{z/(i\sqrt{N})} \prod_{d|N} \eta(dz/N)^{r_{N/d}} \\
& = \sqrt{z/(i\sqrt{N})}\eta(\br,z/N).
\end{align*}
Above we used that $\prod_{d|N}d^{-r_d/2}=N^{-1/4}$, which follows from the conditions on $r_d$. Eta-quotients of this type are known to be weakly holomorphic functions of weight $1/2$ on $\Gamma_0(N)$ (see \cite[Prop. 1.39]{Bh}). In particular, we conclude that the function
$$
F(\br,z)=\eta(\br,z/\sqrt{N}) = q^{b/\sqrt{N}}\sum_{n\geq  k}\al_n q^{n/\sqrt{N}}.
$$
satisfies the functional equations 
\begin{align*}
F(\br,z+\sqrt{N})=e^{2\pi i b}F(\br,z) \quad \text{and} \quad
\sqrt{i/z}\,F(\br,-1/z)=F(\br,z)
\end{align*}
for all $z\in \H$. 
Since the function $(\im z)^{1/4} |\eta(\br,z)|$ is $\Gamma_0(N)$-invariant, one can straightforwardly deduce that $|\al_n| \ll n^k$, for some $k>0$, whenever $\eta(\br,z)$ grows at most polynomially at all cusps of $\Gamma_0(N)$ (one can even obtain the Hecke-bound $|\al_n| \ll n^{1/4}$ if $\eta(\br,z)$ vanishes exponentially at all cusps of $\Gamma_0(N)$). This condition turns out to be equivalent to the verification of at most $d(N)$-linear inequalities given by\footnote{We are thankful to D. Radchenko for pointing this out.} \cite[Lemma 1.40]{Bh}. Assuming that such inequalities are satisfied, $F$ satisfies the hypotheses of Lemma \ref{lem:FG}, and the measure
$$
\mu = \sum_{n\geq 0}\al_n (\del_{\sqrt{2(n+b)/\sqrt{N}}} + \del_{-\sqrt{2(n+b)/\sqrt{N}}}) \quad \text{satisfies } \quad \ft \mu=\mu.
$$
It is now obvious that one can build infinitely many examples of self-dual crystalline measures $\mu$ by making linear combination of these types of eta-quotients. 

For instance, consider now the following family
\begin{align*}
\frac{\eta(z)^{24c-2}\eta(4z)^{24c-2}}{\eta(2z)^{48c-5}} & = q^{c}\left(1-(24c - 2)q + (288c^2 - 36c)q^2 + \sum_{n\geq 3} \al_{n,c} q^n \right),
\end{align*}
which fits the previous discussion for $N=4$, $(r_1,r_2,r_4)=(24c-2,5-48c,24c-2)$  and $b=c$. We will have to impose $c\in [0,1/8]$,  as this range guarantees that the coefficients satisfy that Hecke bound $|\al_{n,c}| \ll n^{1/4}$, since it will be holomorphic at the cusps $0$, $i\infty$ and $1/2$ of $\Gamma_0(4)$ (actually, numerically, the coefficients  $\al_{n,c}$ seem to be bounded). The example above produces a measure $\mu_{c}$ such that
$$
\mu_{c}= \ft \mu_{c} \quad \text{and}\quad \supp(\mu_{c})\subset \{\sqrt{c+n}\}_{n\geq 0} \cup \{-\sqrt{c+n}\}_{n\geq 0}.
$$
Note that $\mu_{0}$  is classical Poisson summation and $\mu_{\tfrac19}$ is Guinand's example. Perhaps other interesting constructions may be derived from \cite{Lem} and the references therein. It is not hard to show that $\{ \al_{n,c}\}_{n\geq 0}$ are polynomials in $c$ with integer coefficients and degree $n$. To finish the discussion, note that if $c$ is irrational then $\{\sqrt{c+n}\}_{n\geq 0}$ intersects any infinite arithmetic progression at most $2$ times. Indeed, having three hits implies that
$$
\frac{\sqrt{n_2+c}-\sqrt{n_1+c}}{k_1} = \frac{\sqrt{n_3+c}-\sqrt{n_2+c}}{k_2}
$$ 
is satisfied for some integers $n_3>n_2>n_1\geq 0$ and $k_1,k_2>0$. Writing $A=n_2+c$, $n_3+c=A+k_3$, $n_1+c=A-k_4$ and solving for $A$, shows that $A$ is rational, which is absurd.

Such measures are also connected with new Fourier interpolation formulae produced in \cite{RV}, which as a byproduct produces crystalline pairs with space and spectral supports contained in $\{\pm \sqrt{n}\}_{n\geq 0}\cup \{x_0\}$, for any $x_0$ not an integer square root.  These nodes can also be perturbed to be $\{\pm \sqrt{n+\ep_n}\}_{n\geq 0}$, as long as $\ep_0=0$ and $|\ep_n|\leq \delta n^{-5/4}$ for some small $\delta>0$, see \cite{RS1}. The recent  papers \cite{RS,KNS} also suggest that there should exist a crystalline pair with space and spectral supports contained in $\{\pm(\tfrac1{2\al} n)^{\al}\}_{n\geq 0}$ and $\{\pm(\tfrac1{2(1-\al)} n)^{1-\al}\}_{n\geq 0}$ respectively, whenever $0<\al<1$. Moreover, \cite[Cor. 3]{KNS} proves the striking result that for every (supercritical) pair of unbounded increasing sequences $\T=\{t_j\}_{j\in \z}$ and $\S=\{s_j\}_{j\in \z}$ such that
$$
\max\left(\limsup_{j\to\infty} |t_j|^{p-1}(t_{j+1}-t_j),\limsup_{j\to\infty} |s_j|^{q-1}(s_{j+1}-s_j)\right) <\frac12,
$$
for some $1<p,q<\infty$ with $\frac{1}{p}+\frac{1}{q}=1$, there is a $\fs$-pair $(\mu,a)$ with $\mu$ supported in $\T$ and $\a$ supported in $\S$. An important simple example is again $\T=\{\pm (\tfrac{a}{2\al} n)^{\al}\}_{n \geq 0}$ and $\S=\{\pm (\tfrac{b}{2\be}n)^{\be}\}_{n\geq 0}$ for $0<\al,\be,a,b<1$ with $\al+\be\leq 1$ (the multiplying constants $a$ and $b$ are irrelevant unless $\al+\be=1$, in which case we must impose $ab<1$, but conjecturally $ab=1$ is possible). In fact, $(\T,\S)$ is both a Fourier uniqueness pair and has an associated reconstruction/interpolation formula (for more info see \cite{KNS}).

\subsection{Almost Crystalline Measures}\label{sec:ACM}
We say $(\mu,a)$ is an almost Fourier summation pair ($\afs$-pair) if $\mu$ is strongly tempered, $\a$ is locally summable function and 
$$
(\ft \mu -a)|_{C_c^\infty(\r)}=\nu |_{C_c^\infty(\r)},
$$
for some absolutely continuous complex-measure $\nu$.  In particular, if $\p\in C_c^\infty(\r)$ we have
$$
\sum_{\la\in \r} a(\la) \p(\la)   + \int_\r  \p(t) \d \nu(t)   = \int_\r  \ft \p(t) \d \mu(t).
$$
One way to think about this that $\nu$ is the absolutely continuous part of $\ft \mu$, while $a(\cdot)$ is the (atomic) singular part. We then let the class $\acm$ (Almost Crystalline Measures) denote the $\afs$-pairs $(\mu,a)$ such that both $\mu$ and $\a$ have locally finite support.  The typical example is Guinand's prime summation formula (also known as Riemann–Weil explicit formula), which states that for every $\p\in C_c^\infty(\r)$ we have
\begin{align*}
& \frac{\log (\pi/N)}{2\pi}\p(0) + \frac{1}{2\pi}\sum_{n\geq 2} \frac{\Lambda(n)\chi(n)}{\sqrt{n}}\bigg(\p\left(\frac{\log n}{2\pi}\right) + \p\left(\frac{-\log n}{2\pi}\right)\bigg) - 2\int_\r \p(t) \cosh(\pi t) \d t  \\
& = -\sum_{\ga}  \ft \p(\ga) + \frac{1}{2\pi } \int_\r \ft \p (t) \Re \psi(\tfrac14+i\tfrac{t}2)\d t
\end{align*}
where $\psi(z)=\Gamma'(z)/\Gamma(z)$ is the digamma function, $\Lambda(n)$ is the von Mangoldt function, $\rho=1/2+i\ga$ are the zeros of the Dirichlet $L$-function $$L(\chi,s)=\sum_{n\geq 1}\chi(n)n^{-s}$$  on the critical strip $0<\Re \rho<1$ and $\chi$ is a primitive, real and even Dirichlet character of modulus $N$. Since $|\psi(1/4+it/2)| \ll \log (2+|t|)$ and $\sum_{\ga} \frac{1}{1+\ga^2} < \infty$, we conclude, assuming the Riemann Hypothesis for $L(\chi,s)$, that if we let
\begin{align*}
a&=\frac{\log (\pi/N)}{2\pi}\1_0 + \frac{1}{2\pi}\sum_{n\geq 1} \frac{\Lambda(n)\chi(n)}{\sqrt{n}}(\1_{\frac{\log n}{2\pi}} + \1_{-\frac{\log n}{2\pi}} ) \\
\nu &= -2\cosh(\pi t)\d t  \\
\mu &= -\sum_{\ga}  \del_{\ga} +  \frac{1}{2\pi}\Re \psi(\tfrac14+i\tfrac{t}2)\d t
\end{align*}
then $(\mu,a) \in \acm$. The difference of any two of such pairs above for two characters $\chi_1$ and $\chi_2$ real, even and primitive, but different moduli $N$, produces a pair in $\cm$, this was pointed out first by Guinand \cite{Gui}, and we explain how this construction fits our framework in the end of Remark \ref{rem:genamu}. Such measures are also connected with the remarkable new Fourier interpolation formulae produced in \cite{BRS}. Assuming all non-trivial zeros of $L(\chi,s)$ have real part $1/2$ and are simple, these new formulae produce $\cm$-pairs $(\mu,a)$ with $\supp(\mu)=\{\ga \in \r : L(\chi,1/2+i\ga)=0\}\cup\{\ga_0\}$ and $\supp(a)=\{ \sgn(n)\frac{\log |n|}{4\pi}\}_{n\geq 1}$, where $\ga_0$ is any real such that $L(\chi,1/2+i\ga_0)\neq 0$.

\section*{Acknowledgments}
We thank Oleksiy Klurman and Bryce Kerr for all the conversations we had around this topic during their visit to Rio. Theorems do come out of the smoke. We thank Friedrich Littmann for the fruitful conversations and comments which inspired Theorem 5. We also thank Michael Baake and Peter Sarnak for further discussions, references and remarks after the first version of this paper. The author acknowledges support from the following funding agencies: The Office of Naval Research GRANT14201749 (award number N629092412126), The Serrapilheira Institute (Serra-2211-41824), FAPERJ (E-26/200.209/2023) and CNPq (309910/2023-4).



\begin{thebibliography}{100}

\bibitem{ACV}
L. Alon, A. Cohen and C. Vinzant,
Every real-rooted exponential polynomial is the restriction of a Lee-Yang polynomial, arXiv:2303.03201.

\bibitem{AV}
L. Alon and C. Vinzant,
Gap distributions of Fourier quasicrystals via Lee-Yang polynomials, arXiv:2307.13498.

\bibitem{AP}
L. Amerio and G. Prouse, Almost-periodic functions and functional equations, Van Nostrand Reinhold Company, New York, 1971.

\bibitem{BS}
{M. Baake and N. Strungaru},
A note on tempered measures, Colloquium Mathematicum 172 (2023), 15-30.




\bibitem{Bar2} A.\ Baranov. Differentiation in de Branges spaces and embedding theorems. J.\ Math.\ Sci. 101 no.\ 2, (2002), 2881 -- 2913.

\bibitem{Bar1} A.\ Baranov. Estimates of the $L^p$-norms of derivatives in spaces of entire functions. J.\ Math.\ Sci.\ 129 no.\ 4, (2005). 3927 -- 3943.

\bibitem{BG1} M. Baake and U. Grimm (eds), Aperiodic Order Volume 1: A Mathematical Invitation, Cambridge University Press, October 2017.

\bibitem{BG2} M. Baake and U. Grimm (eds), Aperiodic Order Volume 2: Crystallography and Almost Periodicity, Cambridge University Press, October 2017.

\bibitem{BST}
M. Baake, N. Strungaru and V. Terauds,
Pure point measures with sparse support and sparse Fourier--Bohr support, Trans. London Math. Soc. 7 (2020), 1-32.




\bibitem{Be}
A. S. Besicovitch, Almost Periodic Functions, 
Dover Publications, 1954 - Fourier series - 180 pages.

\bibitem{Bh}
S. Bhattacharya, Factorization of holomorphic eta quotients, Ph.D thesis, Rheinische Friedrich-Wilhelms-Universit\"at Bonn, 2014.

\bibitem{BH}
V. Blomer and G. Harcos. Hybrid bounds for twisted L-functions. J. Reine Angew. Math., 621:53–79, 2008.


  
  \bibitem{dB} L.\ de Branges. Hilbert spaces of entire functions. Prentice Hall, Englewood Cliffs. NJ, 1968.
  

 \bibitem{Boas} R.P.\ Boas, Entire Functions, Academic Press, New York, 1954.


\bibitem{BRS}
A. Bondarenko, D. Radchenko and K. Seip,
Fourier Interpolation with Zeros of Zeta and L-Functions,
Constructive Approximation (2023) 57:405-461.

\bibitem{Bo}
H. Bohr, Zur Theorie der fastperiodischen Funktionen I, Acta Math. 45 (1925), p. 29-127.

\bibitem{CG}
H. Cohn and F. Gon\c{c}alves,
An optimal uncertainty principle in twelve dimensions via modular forms,
Inventiones Mathematicae 217 (2019), 799-831.

\bibitem{Dy} 
F. Dyson, Birds and frogs, Notices Amer. Math. Soc. 56 (2009), no. 2, 212-223.

\bibitem{DW}
S. Deloudi and S. Walter, Crystallography of Quasicrystals: Concepts, Methods and Structures, Springer Series in Materials Science 126 (2012).

\bibitem{Ga}
J. B. Garnett, Bounded Analytic Functions, Graduate Texts in Mathematics, Springer New York, NY., 2007.

\bibitem{GL}
F. Gonçalves and F. Littmann,
Interpolation Formulas with Derivatives in De Branges Spaces II. J. Math. Anal. Appl. 458 (2018), issue 2, p. 1091-1114.


\bibitem{Gui}
A. P. Guinand, Concordance and the harmonic analysis of sequences, Acta Math. 101(3-4): 235-271 (1959).


\bibitem{IG}
A. E. Ingham. A Note on Fourier Transforms. J. London Math. Soc., 9(1):29–32, 1934.

\bibitem{JEH}
K. Jaganathan, Y. C. Eldar and B. Hassibi,
Phase Retrieval: An Overview of Recent Developments,
Optical Compressive Imaging, pp. 264-292 (2016).

\bibitem{Kr}
M.G. Krein, A contribution to the theory of entire functions of exponential type, Bull. Acad. Sci. URSS. S6r. Math. [Izvestiya Akad. Nauk. SSSR] 11 (1947) 309-326.



\bibitem{KS}
P. Kurasov and P. Sarnak, Stable polynomials and crystalline measures, Journal of Mathematical Physics, 61(8):083501, 2020.

\bibitem{KNS}
A. Kulikov, F. Nazarov and M. Sodin,
Fourier uniqueness and non-uniqueness pairs, arXiv:2306.14013v1.

\bibitem{La}
J. C. Lagarias, Mathematical quasicrystals and the problem of diffraction. In Directions in Mathematical Quasicrystals, Baake M. and Moody R.V. (eds.), (2000), pp. 161-193 (AMS, Providence, RI).

\bibitem{Lem}
R.J. Lemke Oliver, 
Eta-quotients and theta functions,
Advances in Mathematics 241 (2013) 1-17.

\bibitem{LO} N. Lev and A. Olevskii, Fourier quasicrystals and discreteness of the diffraction spectrum, Adv. Math. 315 (2017), 1-26.

\bibitem{LO2} N. Lev and A. Olevskii, Quasicrystals and Poisson’s summation formula. Invent. math. 200, 585–606 (2015). 

\bibitem{L}
B. Ja. Levin,
Distribution of Zeros of Entire Functions, Translations of Mathematical Monographs Vol. 5, American Mathematical Soc., 1964.

\bibitem{L2}
B. Ja. Levin,
Lectures on Entire Functions, Translations of Mathematical Monographs 150, American Mathematical Society 1996.

\bibitem{LZ}
B.M. Levitan and V.V. Zhikov,
Almost periodic functions and differential equations, Cambridge University Press, 1983.


\bibitem{Me} Y. Meyer, Nombres de Pisot, nombres de Salem et analyse harmonique, Lecture Notes in Mathematics, Vol. 117, Springer-Verlag, Berlin-New York, 1970 (French). Cours Peccot donne au Coll`ege de France en avril-mai 1969. 

\bibitem{Me95} Y. Meyer, Quasicrystals, Diophantine approximation and algebraic numbers
Beyond quasicrystals (Les Houches, 1994), Springer, Berlin 1995.

\bibitem{Me16} Y. Meyer, Measures with locally finite support and spectrum, Proc. Natl. Acad. Sci. USA 113 (2016), 3152-3158.

\bibitem{Me12} Y. Meyer, Quasicrystals, Almost Periodic Patterns, Mean-periodicFunctions and Irregular Sampling,
Afr. Diaspora J. Math. (N.S.) 13 (1), 1-45, (2012)

\bibitem{Mo} R.V. Moody, Model Sets: A Survey, In: Axel, F., Dénoyer, F., Gazeau, JP. (eds) From Quasicrystals to More Complex Systems. Centre de Physique des Houches, vol 13. Springer, Berlin, Heidelberg. 

\bibitem{OU}
A. Olevskii and A. Ulanovskii, Fourier quasicrystals with unit masses, Comptes Rendus, Mathématique, 358(11-12):1207–1211, 2020.

\bibitem{OU2}
A. Olevskii and A. Ulanovskii, A simple crystalline measure, arXiv:2006.12037v2 .

\bibitem{RV}
D. Radchenko and M. Viazovska, Fourier interpolation on the real line, Publ. Math. IHES 129 (2019), 51–81.

\bibitem{RS}
J. Ramos and M. Sousa, Fourier uniqueness pairs of powers of integers, J. Eur. Math. Soc. (JEMS) 24 (2022), no. 12, 4327-4351.

\bibitem{RS1}
J. Ramos and M. Sousa, Perturbed interpolation formular and applications, To appear in Analysis \& PDE.

\bibitem{QS1D}
S. Ritsch, O. Radulescu, C. Beeli, D.H. Warrington, R. Luck and K. Hiraga, A stable one-dimensional quasicrystal related to decagonal Al-Co-Ni, Philosophical Magazine Letters 
Volume 80, 2000 - Issue 2, Pages 107-118.

\bibitem{dan}
D. Shechtman, I. Blech, D. Gratias, and J. W. Cahn,
Metallic Phase with Long-Range Orientational Order and No Translational Symmetry Phys. Rev. Lett. 53, 1951 (1984).


\end{thebibliography}
\end{document}